\documentclass[preprint,1pt]{elsarticle}
\usepackage{url}
\usepackage{soul}
\usepackage{xcolor}
\usepackage{amssymb}
\usepackage{amsmath}
\usepackage{stmaryrd}
\usepackage{siunitx}
\usepackage{enumitem}
\usepackage{commath}
\usepackage{amsfonts}
\usepackage{mathtools}
\usepackage{amsopn}
\usepackage{amsthm}
\usepackage{graphicx,subfig}
\usepackage{hyperref}
\hypersetup{
	colorlinks   = true, 
	urlcolor     = blue, 
	linkcolor    = blue, 
	citecolor   = red 
}
\usepackage[noabbrev]{cleveref}

\numberwithin{equation}{section}
\newtheorem{theorem}{Theorem}[section]
\newtheorem{lemma}[theorem]{Lemma}

\newtheorem{corollary}[theorem]{Corollary}
\theoremstyle{definition}

\theoremstyle{remark}

\makeatletter
\newcommand{\tnorm}{\@ifstar\@tnorms\@tnorm}
\newcommand{\@tnorms}[1]{%
  \left|\mkern-1.5mu\left|\mkern-1.5mu\left|
   #1
  \right|\mkern-1.5mu\right|\mkern-1.5mu\right|
}
\newcommand{\@tnorm}[2][]{%
  \mathopen{#1|\mkern-1.5mu#1|\mkern-1.5mu#1|}
  #2
  \mathclose{#1|\mkern-1.5mu#1|\mkern-1.5mu#1|}
}
\newcommand{\vertiii}[2][1]{\abs[#1]{\kern0.15ex\norm[#1]{#2}\kern-0.25ex}}

\makeatother

\newcommand{\bp}{\boldsymbol{p}}
\newcommand{\bph}{\boldsymbol{p}_h}

\newcommand{\bq}{\boldsymbol{q}}
\newcommand{\bqh}{\boldsymbol{q}_h}

\newcommand{\br}{\boldsymbol{r}}
\newcommand{\brh}{\boldsymbol{r}_h}

\newcommand{\bs}{\boldsymbol{s}}
\newcommand{\bsh}{\boldsymbol{s}_h}

\newcommand{\sphi}{\phi}
\newcommand{\sphih}{\phi_h}
\newcommand{\smu}{\mu}
\newcommand{\smuh}{\mu_h}
\newcommand{\spsi}{\psi}
\newcommand{\spsih}{\psi_h}

\newcommand{\fphi}{\bar{\phi}}
\newcommand{\fphih}{\bar{\phi}_h}
\newcommand{\fmu}{\bar{\mu}}
\newcommand{\fmuh}{\bar{\mu}_h}

\newcommand{\fpsi}{\bar{\psi}}
\newcommand{\fpsih}{\bar{\psi}_h}

\newcommand{\testp}{\tilde{\boldsymbol{p}}}
\newcommand{\testq}{\tilde{\boldsymbol{q}}}
\newcommand{\testr}{\tilde{\boldsymbol{r}}}
\newcommand{\tests}{\tilde{\boldsymbol{s}}}

\newcommand{\testphi}{\tilde{\phi}}
\newcommand{\testmu}{\tilde{\mu}}
\newcommand{\testpsi}{\tilde{\psi}}

\newcommand{\testfphi}{\theta_\phi}
\newcommand{\testfmu}{\theta_\mu}
\newcommand{\testfpsi}{\theta_\psi}

\newcommand{\elementprod}[3]{\del[#1]{#2 , #3 }_{\mathcal{T}_h}}
\newcommand{\faceprod}[2]{\langle #1 , #2 \rangle_{\partial \mathcal{T}_h}}

\usepackage{xspace,color}

\begin{document}

\begin{frontmatter}
\title{An unconditionally stable hybridizable-embedded discontinuous Galerkin method for the phase field crystal equation}
\author[1]{Giselle Saylor\corref{cor1}}
\ead{gsaylor@oakland.edu}

\author[1]{Tam\'{a}s L. Horv\'{a}th}
\ead{thorvath@oakland.edu}

\author[2]{Natasha S. Sharma\fnref{fn1}}
\ead{nssharma@utep.edu}

\cortext[cor1]{Corresponding author}
\fntext[fn1]{The third author was partially supported by the National Science Foundation under grant NSF DMS-2110774.}

\affiliation[1]{organization={Department of Mathematics and Statistics, Oakland University},
	addressline={146 Library Drive},
	postcode={48309},
	city={Rochester, MI},
	country={USA}}
\affiliation[2]{organization={Department of Mathematical Sciences, The University of Texas at El Paso},
	addressline={500 W University Ave},
	postcode={79968-0514},
	city={El Paso, TX},
	country={USA}}

\begin{abstract}
	This paper presents a first-order convex splitting hybridizable/embedded discontinuous Galerkin method for the phase field crystal equation written in mixed form. Since the sixth-order phase field crystal equation is rewritten as a first-order system, our scheme avoids the calculation of high-order derivatives, which can be computationally expensive. The proposed method uses continuous facet unknowns and static condensation, which significantly reduces the number of coupled degrees of freedom. Using stabilization parameters that satisfy a simple and explicit relation, we show that our scheme is unconditionally energy stable. Moreover, we show the existence and uniqueness of the discrete solution for the case of variable mobility. The scheme's performance and properties are demonstrated through several numerical examples, including benchmark results that align with the existing literature, as well as a comparison of degrees of freedom against other methods.
\end{abstract}


\end{frontmatter}

\section{Introduction}
The phase field crystal (PFC) equation is a continuum model that effectively captures the microstructure evolution of two-phase systems on atomic length and diffusive time scales. This model bridges the gap between atomistic simulations which resolve atomic-level details but are computationally demanding and classical phase field models which lack the atomic-level resolution. Given the model's ability to capture key features in the crystal growth process, it finds applications in several interesting phenomena such as grain growth, liquid phase epitaxial growth, spinodal decomposition, and crack propagation~\cite{EG:04:PFCplastic}. The key features captured by this two-phase system can be described by introducing an atomically varying phase field that is relatively uniform in space except near interfaces, where a rapid change in the phase field occurs~\cite{EG:04:PFCplastic}. Depending on the phenomenon, the phase field could represent concentration or density.
Mathematically, the PFC equation is a sixth-order, nonlinear parabolic partial differential equation in the unknown phase field $\phi$ assumed to be conserved. For a convex, polygonal (polyhedral) domain $\Omega \subset \mathbb{R}^d, \ d = 2,3$ with Lipschitz boundary $\partial \Omega$, the dynamics of a conserved field $\phi$ is given by~\cite{EKHG:02:crystal}
\begin{subequations}
	\begin{align}
		\partial_t \phi &= \nabla \cdot (M(\phi) \nabla \mu),\label{eq:model_pde_a}
		\\
		\mu &= f(\phi) + 2\Delta \phi + \Delta^2 \phi
		\quad \text{ on } \Omega \times (0,T), \label{eq:model_pde_b}
	\end{align}
	\label{eq:model_pde}
\end{subequations}
where the scalar function $M(\phi) \geq 0$ is the mobility, $f(\phi)=\phi^3 + (1-\varepsilon)\phi$, and $\varepsilon < 1$ is the undercooling parameter. The above equation is subject to either periodic or the following homogeneous Neumann boundary conditions
\begin{align}
	\nabla \phi \cdot \boldsymbol{n} = \nabla \Delta \phi \cdot \boldsymbol{n} =  \nabla \mu \cdot \boldsymbol{n} = 0 \quad \text{ on } \partial \Omega \times (0,T).  \label{eq:neumann-bc}
\end{align}

The origins of this model can be traced back to the pioneering works by Elder and co-authors in~\cite{EG:04:PFCplastic,EKHG:02:crystal} where Fourier transformations were used to approximate the solution. Other early numerical approaches can be attributed to the works of Cheng and Warren in~\cite{cheng2008efficient} and Mellenthin et al. in~\cite{mellenthin2008phase}, where a spectral approach was adopted. Further numerical methods include an unconditionally stable finite difference framework developed by Wise and co-authors~\cite{DFWWZ:18:PFCconv, HWWL:09:PFC, WWL:09:PFC}. Finite element methods were proposed in~\cite{BRV:07:PFC, Pei2019}, while an isogeometric analysis framework was used in~\cite{GN:2012:PFC, vignal2015energy}. Another work relying on the mixed system of three second-order equations is~\cite{yang-han2017}, although the spatial discretization uses the Fourier-spectral method.

A $C^0$ interior penalty discontinuous Galerkin (IPDG) discretization was introduced in \cite{DS:23:C0PFC}. The sixth-order equation was discretized as a system of two equations, as given in \cref{eq:model_pde}, and an IPDG discretization was used for the fourth-order equation,  \cref{eq:model_pde_b}. The interior penalty nature of the discretization required a penalty parameter, which must be sufficiently large to ensure coercivity; however, no analytical bound is available. Furthermore, the discretization involved calculating the Hessian of the basis functions which is computationally expensive \cite{DS:23:C0PFC}.

Due to the challenge of numerically approximating higher-order spatial derivatives inherent in the model, an alternative approach involves writing the model in its mixed form as a system of first-order equations. In \cite{guo2016ldg}, a local discontinuous Galerkin (LDG) method is used for the discretization in combination with a semi-implicit time scheme. The paper showed unconditional energy stability; however, there is no proof of well-posedness. In the LDG discretization, the finite element space consists of piecewise polynomials that can be discontinuous across mesh elements, and the unknowns are connected via numerical fluxes defined over the interfaces. The sixth-order equation is rewritten into eight equations using four vector-valued unknowns and four scalar-valued unknowns. It is important to note that some of the equations are purely algebraic, but they further increase the number of coupled degrees of freedom, which is already high due to the discontinuous nature of the basis functions.

A possible way to reduce the number of coupled degrees of freedom is to use a hybridizable discontinuous Galerkin (HDG) method. These methods introduce additional unknowns on the skeleton (edges or faces) of the mesh and then define the numerical fluxes in a way that there is no coupling between neighboring elements, but rather between element and skeleton unknowns \cite{Cockburn:2008,Cockburn:2009,Nguyen:2009}. This allows for static condensation: all the element unknowns (vector- and scalar-valued) can be eliminated from the linear system, and a reduced system can be solved only for the skeleton unknowns. Then, the element unknowns can be recovered in a simple post-processing step, element-wise. The HDG methods have been applied to many problems written as first-order systems, such as for the Navier--Stokes equations in  \cite{cesmelioglu2017analysis}, linear and nonlinear elasticity \cite{kabaria2015hybridizable,soon2009hybridizable}, the Cahn--Hilliard equations in \cite{chen:23:HDG-CH,Kirk:2024}, among others. While DG methods are criticized for having a large number of degrees of freedom, HDG can address this issue; however, for some problems, this advantage is only achieved for high polynomial degrees. To further reduce the number of degrees of freedom, it is possible to use continuous facet functions, which will lead to the so-called embedded discontinuous Galerkin (EDG) methods. The EDG methods reduce the number of globally coupled degrees of freedom to that of a continuous Galerkin problem, independently of the polynomial degree, while still preserving many advantageous properties of the DG methods. The EDG methods have been applied to second-order elliptic problems \cite{Cockburn:2009_EDG}, for a wide variety of computational fluid dynamics problems in \cite{Nguyen:2015}, for incompressible fluid problems, \cite{Cesmelioglu:2020,Rhebergen:2020}, and magnetohydrodynamics \cite{Chen:2024}. 

In this work, we introduce an HDG/EDG method for the discretization of the PFC equation. Using four vector-valued and three scalar unknowns, we rewrite the sixth-order equation into a first-order system, thus avoiding the calculation of Hessians. Similar to \cite{guo2016ldg}, one of the vector-valued unknowns is purely algebraic. Furthermore, we will introduce three skeleton variables and reduce the globally coupled degrees of freedom to those. These skeleton variables can either be discontinuous, which is the HDG discretization, or continuous, leading to the EDG discretization. The proposed method involves stabilization parameters, but we will show that if those parameters satisfy a certain relation, then the semi-discrete and fully discrete HDG/EDG schemes for the PFC \cref{eq:model_pde} are unconditionally energy stable. The relation between the stabilization parameters is explicit, and in contrast to the penalty parameter in the $C^0$ IPDG discretization which oftentimes needs to be very large, the stabilization parameters in our method need only to be positive. We compare the number of globally coupled degrees of freedom of the HDG/EDG discretization proposed in this work, against the LDG and $C^0$ IPDG methods existing in the literature. The EDG method reduces the number of degrees of freedom to approximately 12\% of those corresponding to the LDG method. Moreover, we will show well-posedness of the fully discrete problem with variable mobility by recasting it as a fixed-point problem and using Brouwer's fixed-point theorem.

This paper is organized as follows. In Section~\ref{sec:lhdg}, we introduce the HDG/EDG method and establish the energy stability for the semi-discrete scheme and the fully discrete scheme in Section~\ref {sec:energy_stability}. In Section~\ref{sec:existence}, we prove the existence and uniqueness of the discrete solution. We demonstrate the performance of our method with both periodic and Neumann boundary conditions in Section~\ref{sec:numerical_results} through several numerical examples, and close the paper with some concluding remarks in Section~\ref{sec:conclusions}.
\section{The hybridizable/embedded discontinuous Galerkin discretization}
\label{sec:lhdg}

First, we rewrite the problem \cref{eq:model_pde} as a first-order system in the following way
\begin{subequations}
	\begin{align}
		\br &= -\nabla \sphi,
		\\
		\spsi &= -\nabla \cdot \br,
		\\
		\bq &= -\nabla \spsi,
		\\
		\smu &= f(\sphi) + 2\spsi - \nabla \cdot \bq,
		\\
		\bp &= -\nabla \smu,
		\\
		\bs &= M(\sphi)\bp,
		\\
		\partial_t \sphi &= -\nabla \cdot \bs.
		\label{eq:final_eq_new_vars}
	\end{align}
	\label{eq:new_variables}
\end{subequations}
Let $\mathcal{T}_h$ denote a conforming triangulation of $\Omega$ consisting of either simplicial or hexahedral elements with diameter $h_K$. The mesh size is defined by $h = \max_{K\in \mathcal{T}_h}h_K$. We denote by $\mathcal{E}_h$ the set of edges (or faces in three dimensions) of the elements of $\mathcal{T}_h$. For any integer $k \ge 0$, we let $\mathcal{R}^k(S)$ denote the set of polynomials defined on $S\subset \mathbb{R}^s, \ s=1,\ldots, d$ with degree at most $k$ for simplicial elements, or with degree at most $k$ in each variable for hexahedral elements. We define the following finite element spaces
\begin{align}
	W_h &:= \cbr[1]{w_h \in L^2(\mathcal{T}_h) : w_h|_K \in \mathcal{R}^k(K), \forall K \in \mathcal{T}_h},
	\\
	\boldsymbol{V}_h &:= \cbr[1]{\boldsymbol{v}_h \in \sbr[1]{L^2(\mathcal{T}_h)}^d : \boldsymbol{v}_h|_K \in \sbr[1]{\mathcal{R}^k(K)}^d, \forall K \in \mathcal{T}_h},
	\\
	M_h &:= \cbr[1]{\bar{w}_h \in L^2(\mathcal{E}_h) : \bar{w}_h|_K \in \mathcal{R}^k(e), \forall e \in \mathcal{E}_h} \cap C(\mathcal{E}_h)\label{eq:M_space}.
\end{align}
The above definition $M_h$ uses facet spaces that are continuous, leading to an Embedded DG (EDG) discretization. One could consider the space
$$
M_h^{*} := \cbr[1]{\bar{w}_h \in L^2(\mathcal{E}_h) : \bar{w}_h|_K \in \mathcal{R}^k(e), \forall e \in \mathcal{E}_h},
$$
which would use discontinuous facet functions, and would lead to a Hybridized DG (HDG) discretization. All theorems hereafter are valid for both spaces, $M_h$ and $M_h^{*}$.

For simplicity, we denote
\begin{equation}
	\del[1]{f,g}_{\mathcal{T}_h} = \sum_{K \in \mathcal{T}_h} \int_K fg \dif x, \quad \langle f,g\rangle_{\partial \mathcal{T}_h} = \sum_{K \in \mathcal{T}_h} \int_{\partial K} fg \dif s.
\end{equation}
Moreover, we define \begin{equation}
	\norm[1]{f}_{\mathcal{T}_h}^2 = \sum_{K \in \mathcal{T}_h} \int_K f^2 \dif x, \quad
	\norm[1]{f}_{\partial\mathcal{T}_h}^2 = \sum_{K \in \mathcal{T}_h} \int_{\partial K} f^2 \dif s.
\end{equation}

The semi-discrete form in space is obtained in the usual way by testing with appropriate functions in the corresponding discrete spaces, applying integration by parts, and introducing numerical fluxes. 
For almost each $t$ in $[0,T]$,
the semi-discrete problem is: find $(\brh,\bqh,\bph,\bsh,\sphih,\spsih,\smuh,\fpsih,\fmuh,\fphih) \in \boldsymbol{V}_h^4 \times W_h^3 \times M_h^3$ such that the following equations are satisfied for all $(\testr,\testq,\testp,\tests,\testphi,\testpsi,\testmu,\testfpsi,\testfmu,\testfphi)\in \boldsymbol{V}_h^4 \times W_h^3 \times M_h^3$
\begin{subequations}
	\begin{align}
		\elementprod{1}{\brh}{\testr} &= \elementprod{1}{\sphih} {\nabla \cdot \testr} - \faceprod{\fphih}{\testr \cdot \boldsymbol{n}},
		\label{eq:r_semidisc}
		\\
		\elementprod{1}{\spsih}{\testphi} &= -\elementprod{1}{\nabla \cdot \brh}{\testphi} - \faceprod{\widehat{\brh\cdot \boldsymbol{n}} - \brh \cdot \boldsymbol{n}}{\testphi},
		\label{eq:psi_semidisc}
		\\
		\elementprod{1}{\bqh}{\testq} &= \elementprod{1}{\spsih}{\nabla\cdot \testq} - \faceprod{\fpsih}{\testq \cdot \boldsymbol{n}},
		\label{eq:q_semidisc}
		\\
		\elementprod{1}{\smuh}{\testpsi} &= \elementprod{1}{f(\sphih)}{\testpsi} + 2\elementprod{1}{\spsih}{\testpsi} - \elementprod{1}{\nabla \cdot \bqh}{\testpsi}
		\notag
		\\
		&\quad - \faceprod{\widehat{\bqh\cdot \boldsymbol{n}} - \bqh \cdot \boldsymbol{n}}{\testpsi},
		\label{eq:mu_semidisc}
		\\
		\elementprod{1}{\bph}{\testp} &= \elementprod{1}{\smuh}{\nabla \cdot \testp}- \faceprod{\fmuh}{\testp \cdot \boldsymbol{n}},
		\label{eq:p_semidisc}
		\\
		\elementprod{1}{\bsh}{\tests} &= \elementprod{1}{M(\sphih)\bph}{\tests},
		\label{eq:s_semidisc}
		\\
		\elementprod{1}{\partial_t \sphih}{\testmu} &= -\elementprod{1}{\nabla \cdot \bsh}{\testmu} - \faceprod{\widehat{\bsh\cdot \boldsymbol{n}} - \bsh \cdot \boldsymbol{n}}{\testmu},
		\label{eq:phi_semidisc}
		\\
		\faceprod{\widehat{\brh \cdot \boldsymbol{n}}}{\testfphi} &= 0,
		\label{eq:rhat_semidisc}
		\\
		\faceprod{\widehat{\bqh \cdot \boldsymbol{n}}}{\testfpsi} &= 0,
		\label{eq:qhat_semidisc}
		\\
		\faceprod{\widehat{\bsh \cdot \boldsymbol{n}}}{\testfmu} &= 0.
		\label{eq:shat_semidisc}
	\end{align}
	\label{eq:semidiscrete_form}
\end{subequations}
The numerical fluxes $\widehat{\brh\cdot \boldsymbol{n}},\,\widehat{\bqh\cdot \boldsymbol{n}}$ and $\widehat{\bsh\cdot \boldsymbol{n}}$ are given by
\begin{subequations}
	\begin{align}
		\widehat{\brh \cdot \boldsymbol{n}} &= \brh\cdot \boldsymbol{n} + \tau_1(\sphih - \fphih),
		\\
		\widehat{\bqh \cdot \boldsymbol{n}} &= \bqh\cdot \boldsymbol{n} + \tau_2(\spsih - \fpsih) - \tau_4(\sphih - \fphih),
		\\
		\widehat{\bsh \cdot \boldsymbol{n}} &= \bsh\cdot \boldsymbol{n} + \tau_3(\smuh - \fmuh).
	\end{align}
	\label{eq:num_fluxes}
\end{subequations}
\Cref{eq:rhat_semidisc,eq:qhat_semidisc,eq:shat_semidisc} impose continuity of the numerical fluxes across faces of the mesh. Moreover, $\tau_1, \tau_2, \tau_3, \tau_4 > 0$ are stabilization parameters, and in \cref{sec:energy_stability} we will show the relation that they need to satisfy so that the scheme is unconditionally energy stable.
\section{Energy stability}
\label{sec:energy_stability}

In this section, we will show that our scheme is unconditionally energy stable. First, we obtain the energy relation that the model problem satisfies. The energy in the system is given by~\cite{HWWL:09:PFC}:
\begin{equation*}
	E = \int_\Omega \del{\frac{1}{4}\phi^4 + \frac{1-\varepsilon}{2}\phi^2 - \abs[0]{\nabla \phi}^2 + \frac{1}{2}(\Delta\phi)^2}\dif x,
\end{equation*}
which, written in terms of the variables defined in \cref{eq:new_variables}, is
\begin{equation}
	E = \int_\Omega \del{\frac{1}{4}\phi^4 + \frac{1-\varepsilon}{2}\phi^2 - \abs[0]{\br}^2 + \frac{1}{2}\spsi^2}\dif x.
	\label{eq:energy}
\end{equation}

At the continuous level, we have the following energy balance equation
\begin{align*}
	\partial_t E &= \int_\Omega \del{\phi^3\partial_t \phi + \del[1]{1-\varepsilon}\phi \partial_t \phi - 2\nabla \phi \cdot \nabla \partial_t \phi + \Delta\phi \Delta \partial_t \phi}\dif x
	\\
	&= \int_\Omega \del{\phi^3\partial_t \phi + \del[1]{1-\varepsilon}\phi \partial_t \phi + 2\Delta \phi \partial_t \phi + \Delta^2\phi \partial_t \phi}\dif x
	\\
	&= \int_\Omega \nabla \cdot \del[1]{M(\phi)\nabla \mu}\mu \dif x
	\\
	&= -\int_\Omega M(\phi)\abs[1]{\nabla \mu}^2 \dif x,
\end{align*}
where the second and the fourth lines are obtained from integration by parts and using the homogeneous Neumann or the periodic boundary conditions. 

Thus, the energy balance equation is
\begin{equation}
	\partial_t E + \norm[1]{M(\phi)^{1/2}\nabla \mu}^2_{L^2(\Omega)} = 0.
\end{equation}

Next, we show that the semi-discrete problem \cref{eq:semidiscrete_form} is energy stable.
\begin{lemma}[Energy stability of the semi-discrete problem]
	If $\tau_1 > 0$, $\tau_2 = \tau_1$, $\tau_3 > 0$ and $\tau_4 = 2\tau_1$ in the definition of the numerical fluxes \cref{eq:num_fluxes}, then the semi-discrete scheme satisfies the following energy balance
	\begin{equation}
		\partial_t E_h + \norm[1]{M(\sphih)^{1/2}\bph}_{\mathcal{T}_h}^2 \leq 0,
	\end{equation}
	where the semi-discrete energy $E_h$ is given by
	\begin{equation}
		E_h = \sum_{K\in \mathcal{T}_h}\int_K \del{\frac{1}{4}\sphih^4 + \frac{1-\varepsilon}{2}\sphih^2 - \abs[0]{\brh}^2 + \frac{1}{2}\spsih^2}\dif x.
	\end{equation}
\end{lemma}

\begin{proof}
	Take the time derivative of \cref{eq:r_semidisc}:
	\begin{multline}
		\elementprod{1}{\partial_t \brh}{\testr} + \elementprod{1}{\brh}{\partial_t \testr} = \elementprod{1}{\partial_t \sphih} {\nabla \cdot \testr} + \elementprod{1}{\sphih} {\nabla \cdot (\partial_t\testr)} 
		\\
		- \faceprod{\partial_t \fphih}{\testr \cdot \boldsymbol{n}}- \faceprod{\fphih}{\partial_t\testr \cdot \boldsymbol{n}}.
	\end{multline}
	Since \cref{eq:r_semidisc} holds true for all $\testr \in \boldsymbol{V}_h$, then it also holds true for $\partial_t\testr$. Thus, the previous equation becomes
	\begin{equation}
		\elementprod{1}{\partial_t \brh}{\testr} = \elementprod{1}{\partial_t \sphih} {\nabla \cdot \testr} 
		- \faceprod{\partial_t \fphih}{\testr \cdot \boldsymbol{n}}.
		\label{eq:aux_dt_r}
	\end{equation}
	
	Let $\testr = -2\brh$ and $\testr = -\bqh$ separately in \cref{eq:aux_dt_r}, to obtain
	\begin{subequations}
		\begin{align}
			-2\elementprod{1}{\partial_t \brh}{\brh} &= -2\elementprod{1}{\partial_t \sphih} {\nabla \cdot \brh} 
			+2 \faceprod{\partial_t \fphih}{\brh \cdot \boldsymbol{n}},
			\label{eq:aux1_r}
			\\
			-\elementprod{1}{\partial_t \brh}{\bqh} &= -\elementprod{1}{\partial_t \sphih} {\nabla \cdot \bqh} 
			+ \faceprod{\partial_t \fphih}{\bqh \cdot \boldsymbol{n}}.
			\label{eq:aux2_r}
		\end{align}    
	\end{subequations}
	Similarly, we take the time derivative of \cref{eq:psi_semidisc}, and take $\testphi = \spsih$, to get
	\begin{equation}
		\elementprod{1}{\partial_t \spsih}{\spsih} = -\elementprod{1}{\nabla \cdot \partial_t \brh}{\spsih} - \faceprod{\widehat{\partial_t \brh\cdot \boldsymbol{n}} - \partial_t \brh \cdot \boldsymbol{n}}{\spsih}.
		\label{eq:aux_psi}
	\end{equation}
	Letting $\testphi = -2\partial_t\sphih$, $\testq = \partial_t\brh$, $\testpsi = -\partial_t\sphih$, $\testp = \bsh$, $\tests = -\bph$, $\testmu = \smuh$ in \cref{eq:psi_semidisc} - \cref{eq:phi_semidisc}, adding up together with \cref{eq:aux1_r,eq:aux2_r,eq:aux_psi}, and simplifying, we obtain
	\begin{align}
		\elementprod{1}{f(\sphih)}{\partial_t\sphih} -2\elementprod{1}{\partial_t \brh}{\brh} + \elementprod{1}{\partial_t \spsih}{\spsih} + \elementprod{1}{M(\sphih)\bph}{\bph}
		&=
		\notag
		\\
		+2 \faceprod{\widehat{\brh\cdot \boldsymbol{n}} - \brh \cdot \boldsymbol{n}}{\partial_t\sphih} +2 \faceprod{\partial_t \fphih}{\brh \cdot \boldsymbol{n}}
		+ \faceprod{\widehat{\bqh\cdot \boldsymbol{n}} - \bqh \cdot \boldsymbol{n}}{\partial_t\sphih} 
		\notag
		\\
		+ \faceprod{\partial_t \fphih}{\bqh \cdot \boldsymbol{n}}
		- \faceprod{\widehat{\bsh\cdot \boldsymbol{n}} - \bsh \cdot \boldsymbol{n}}{\smuh} - \faceprod{\fmuh}{\bsh \cdot \boldsymbol{n}} 
		\notag
		\\
		- \faceprod{\widehat{\partial_t \brh\cdot \boldsymbol{n}} - \partial_t \brh \cdot \boldsymbol{n}}{\spsih} - \faceprod{\fpsih}{\partial_t\brh \cdot \boldsymbol{n}} .
	\end{align}
	Using the definition of the numerical fluxes \cref{eq:num_fluxes} on the right hand side, we have
	\begin{align}
		\elementprod{1}{f(\sphih)}{\partial_t\sphih} -2\elementprod{1}{\partial_t \brh}{\brh} + \elementprod{1}{\partial_t \spsih}{\spsih} + \elementprod{1}{M(\sphih)\bph}{\bph}
		&=
		\notag
		\\
		+2 \faceprod{\tau_1(\sphih - \fphih)}{\partial_t\sphih} +2 \faceprod{\partial_t \fphih}{\brh \cdot \boldsymbol{n}}
		\notag
		\\
		+ \faceprod{\tau_2(\spsih - \fpsih) - \tau_4(\sphih - \fphih)}{\partial_t\sphih} + \faceprod{\partial_t \fphih}{\bqh \cdot \boldsymbol{n}}
		\notag
		\\
		- \faceprod{\tau_3(\smuh - \fmuh)}{\smuh} - \faceprod{\fmuh}{\bsh \cdot \boldsymbol{n}} 
		\notag
		\\
		- \faceprod{\tau_1(\partial_t \sphih - \partial_t \fphih)}{\spsih} - \faceprod{\fpsih}{\partial_t\brh \cdot \boldsymbol{n}}.
		\label{eq:aux_energy1}
	\end{align}
	
	Next, we work on the equations for the continuity of the fluxes. Take the time derivative of \cref{eq:rhat_semidisc} and let $\testfphi = -\fpsih$, to get
	\begin{equation}
		-\faceprod{\partial_t\brh\cdot\boldsymbol{n} + \tau_1(\partial_t \sphih - \partial_t \fphih)}{\fpsih} = 0.
		\label{eq:aux_rhat}
	\end{equation}
	Taking $\testfmu = -\fmuh$, $\testfphi = 2\partial_t \fphih$, $\testfpsi = \partial_t \fphih$ in \cref{eq:rhat_semidisc} - \cref{eq:shat_semidisc}, add up to \cref{eq:aux_energy1} and \cref{eq:aux_rhat}, and simplify to obtain
	\begin{align}
		\elementprod{1}{f(\sphih)}{\partial_t\sphih} -2\elementprod{1}{\partial_t \brh}{\brh} + \elementprod{1}{\partial_t \spsih}{\spsih} 
		+ \norm[1]{M^{1/2}(\sphih)\bph}^2_{\mathcal{T}_h} &=
		\notag
		\\
		- \tau_3\norm[1]{\smuh - \fmuh}^2_{\partial\mathcal{T}_h}
		+ \faceprod{\tau_2(\spsih - \fpsih) - \tau_4(\sphih - \fphih)}{\partial_t\sphih - \partial_t \fphih} &
		\notag
		\\
		+ 2 \faceprod{\tau_1(\sphih - \fphih)}{\partial_t\sphih - \partial_t\fphih}
		- \faceprod{\tau_1(\partial_t \sphih - \partial_t \fphih)}{\spsih - \fpsih}.
	\end{align}
	If $\tau_4 = 2\tau_1$, and $\tau_2 = \tau_1$, then
	\begin{multline}
		\elementprod{1}{f(\sphih)}{\partial_t\sphih} -2\elementprod{1}{\partial_t \brh}{\brh} + \elementprod{1}{\partial_t \spsih}{\spsih} + \norm[1]{M^{1/2}(\sphih)\bph}^2_{\mathcal{T}_h} =
		\\
		- \tau_3\norm[1]{\smuh - \fmuh}^2_{\partial\mathcal{T}_h}.
	\end{multline}
	
	Note that 
	\begin{align*}
		\elementprod{1}{f(\sphih)}{\partial_t\sphih} &-2\elementprod{1}{\partial_t \brh}{\brh} + \elementprod{1}{\partial_t \spsih}{\spsih} 
		\\
		&= \sum_{K\in \mathcal{T}_h} \int_K \del{\del[1]{\sphih^3 + (1-\varepsilon)\sphih}\partial_t \sphih - 2 \partial_t\brh\cdot \brh + \spsih \partial_t \spsih}\dif x
		\\
		&= \sum_{K\in \mathcal{T}_h} \partial_t\int_K \del{\frac{1}{4}\sphih^4 + \frac{1-\varepsilon}{2}\sphih^2 - \abs[0]{\brh}^2 + \frac{1}{2}\spsih^2}\dif x
		\\
		&= \partial_t E.
	\end{align*}
	The result follows.
\end{proof}

\subsection{First-order convex splitting temporal discretization}
We divide the time domain $(0,T)$ into $N$ subintervals each of uniform length $\Delta t$ and rely on the notation of $\phi^n$ to denote $\phi(t^n)$ i.e., $\phi$ at time $t^n$. Let $\delta^n_{\Delta t}$ denote the discrete time derivative, e.g., $\delta^n_{\Delta t} \sphih := \del[1]{\sphih^n - \sphih^{n-1}}/\Delta t$. We use a first-order convex splitting, which is described next.

The fully discrete problem is: given $\sphih^{n-1}$ and $\spsih^{n-1}$, find $(\brh^n, \bqh^n, \bph^n, \bsh^n, \sphih^n, \spsih^n, \allowbreak \smuh^n, \fpsih^n, \fmuh^n, \fphih^n) \in \boldsymbol{V}_h^4 \times W_h^3 \times M_h^3$ such that the following equations are satisfied for all $(\testr,\testq,\testp,\tests,\testphi,\testpsi,\testmu,\testfpsi,\testfmu,\testfphi)\in \boldsymbol{V}_h^4 \times W_h^3 \times M_h^3$
\begin{subequations}
	\begin{align}
		\elementprod{1}{\brh^n}{\testr} &= \elementprod{1}{\sphih^n} {\nabla \cdot \testr} - \faceprod{\fphih^n}{\testr \cdot \boldsymbol{n}},
		\label{eq:r_disc}
		\\
		\elementprod{1}{\spsih^n}{\testphi} &= -\elementprod{1}{\nabla \cdot \brh^n}{\testphi} - \faceprod{\widehat{\brh^n\cdot \boldsymbol{n}} - \brh^n \cdot \boldsymbol{n}}{\testphi},
		\label{eq:psi_disc}
		\\
		\elementprod{1}{\bqh^n}{\testq} &= \elementprod{1}{\spsih^n}{\nabla\cdot \testq} - \faceprod{\fpsih^n}{\testq \cdot \boldsymbol{n}},
		\label{eq:q_disc}
		\\
		\elementprod{1}{\smuh^n}{\testpsi} &= \elementprod{1}{f(\sphih^n)}{\testpsi} + 2\elementprod{1}{\spsih^{n-1}}{\testpsi} 
		\label{eq:mu_disc}
		\\
		&\quad - \elementprod{1}{\nabla \cdot \bqh^n}{\testpsi} - \faceprod{\widehat{\bqh^n\cdot \boldsymbol{n}} - \bqh^n \cdot \boldsymbol{n}}{\testpsi},
		\nonumber
		\\
		\elementprod{1}{\bph^n}{\testp} &= \elementprod{1}{\smuh^n}{\nabla \cdot \testp}- \faceprod{\fmuh^n}{\testp \cdot \boldsymbol{n}},
		\label{eq:p_disc}
		\\
		\elementprod{1}{\bsh^n}{\tests} &= \elementprod{1}{M(\sphih^{n-1})\bph^n}{\tests},
		\label{eq:s_disc}
		\\
		\elementprod{1}{\delta_{\Delta t}^n \sphih}{\testmu} &= -\elementprod{1}{\nabla \cdot \bsh^n}{\testmu} - \faceprod{\widehat{\bsh^n\cdot \boldsymbol{n}} - \bsh^n \cdot \boldsymbol{n}}{\testmu},
		\label{eq:phi_disc}
		\\
		\faceprod{\widehat{\brh^n \cdot \boldsymbol{n}}}{\testfphi} &= 0,
		\label{eq:rhat_disc}
		\\
		\faceprod{\widehat{\bqh^n \cdot \boldsymbol{n}}}{\testfpsi} &= 0,
		\label{eq:qhat_disc}
		\\
		\faceprod{\widehat{\bsh^n \cdot \boldsymbol{n}}}{\testfmu} &= 0.
		\label{eq:shat_disc}
	\end{align}
	\label{eq:discrete_form}
\end{subequations}

We remark that given the initial condition $\sphih^0$, $\spsih^0$ can be calculated element-wise as $\spsih^0|_K = \Delta \sphih^0|_K$, unless stated otherwise.
In the following lemma, we show that the fully discrete scheme is unconditionally energy stable.
\begin{lemma}[Fully-discrete energy stability]
	\label{lem:full_energy_stability}
	If $\tau_1 > 0$, $\tau_2 = \tau_1$, $\tau_3 > 0$, and $\tau_4 = 2\tau_1$ in the definition of the numerical fluxes \cref{eq:num_fluxes}, then the fully discrete scheme \cref{eq:discrete_form} is energy stable, so that
	\begin{equation}
		E_h^n + \Delta t \norm[1]{M(\sphih^{n-1})^{1/2}\bph^n}^2_{\mathcal{T}_h} \leq E_h^{n-1},
	\end{equation}
	where
	\begin{equation}
		E_h^n = \sum_{K\in \mathcal{T}_h}\int_K\del{\frac{1}{4}\del[0]{\sphih^n}^4 + \frac{1-\varepsilon}{2}\abs[0]{\sphih^n}^2 -\abs[0]{\brh^n}^2 + \frac 12 \abs[0]{\spsih^n}^2}\dif x.
	\end{equation}
\end{lemma}

\begin{proof}
	Let $\testq = \delta_{\Delta t}^n \brh$, $\testpsi = -\delta_{\Delta t}^n\sphih$, $\testp = \bsh^n$, $\tests = -\bph^n$, $\testmu = \smuh^n$ in \cref{eq:discrete_form}, and add up, to obtain
	\begin{align}
		\elementprod{1}{f(\sphih^n)}{\delta_{\Delta t}^n\sphih} + \elementprod{1}{\bqh^n}{\delta_{\Delta t}^n \brh} + \elementprod{1}{M(\sphih^{n-1})\bph^n}{\bph^n} 
		= 
		\elementprod{1}{\spsih^n}{\nabla\cdot \delta_{\Delta t}^n \brh} 
		\notag
		\\
		- 2\elementprod{1}{\spsih^{n-1}}{\delta_{\Delta t}^n\sphih} 
		+ \elementprod{1}{\nabla \cdot \bqh^n}{\delta_{\Delta t}^n\sphih} 
		- \faceprod{\fpsih^n}{\delta_{\Delta t}^n \brh \cdot \boldsymbol{n}} 
		\notag
		\\+ \faceprod{\widehat{\bqh^n\cdot \boldsymbol{n}} - \bqh^n \cdot \boldsymbol{n}}{\delta_{\Delta t}^n\sphih} 
		- \faceprod{\fmuh^n}{\bsh^n \cdot \boldsymbol{n}}
		- \faceprod{\widehat{\bsh^n\cdot \boldsymbol{n}} - \bsh^n \cdot \boldsymbol{n}}{\smuh^n}
		\label{eq:aux_energy_discrete_1}
	\end{align}
	Since \cref{eq:discrete_form} is satisfied for all $n$, we have
	\begin{subequations}
		\begin{align}
			\elementprod{1}{\brh^{n-1}}{\testr} &= \elementprod{1}{\sphih^{n-1}} {\nabla \cdot \testr} - \faceprod{\fphih^{n-1}}{\testr \cdot \boldsymbol{n}},
			\label{eq:r_disc_n1}
			\\
			\elementprod{1}{\delta_{\Delta t}^n \brh}{\testr} &= \elementprod{1}{\delta_{\Delta t}^n\sphih} {\nabla \cdot \testr} - \faceprod{\delta_{\Delta t}^n\fphih}{\testr \cdot \boldsymbol{n}},
			\label{eq:deltar_disc_n1}
			\\
			\elementprod{1}{\spsih^{n-1}}{\testphi} &= -\elementprod{1}{\nabla \cdot \brh^{n-1}}{\testphi} - \faceprod{\widehat{\brh^{n-1}\cdot \boldsymbol{n}} - \brh^{n-1} \cdot \boldsymbol{n}}{\testphi},
			\label{eq:psi_disc_n1}
			\\
			\elementprod{1}{\delta_{\Delta t}^n\spsih}{\testphi} &= -\elementprod{1}{\nabla \cdot \delta_{\Delta t}^n\brh}{\testphi} - \faceprod{\widehat{\delta_{\Delta t}^n\brh\cdot \boldsymbol{n}} - \delta_{\Delta t}^n\brh \cdot \boldsymbol{n}}{\testphi},
			\label{eq:deltapsi_disc}
		\end{align}
	\end{subequations}
	Let $\testr = -2\delta_{\Delta t}^n\brh$ in \cref{eq:r_disc_n1}, $\testr = -\bqh^n$ in \cref{eq:deltar_disc_n1}, $\testphi = -2\delta_{\Delta t}^n\sphih$ in \cref{eq:psi_disc_n1} and $\testphi = \spsih^n$ in \cref{eq:deltapsi_disc}, and add up together with \cref{eq:aux_energy_discrete_1}, to obtain
	\begin{align}
		\elementprod{1}{f(\sphih^n)}{\delta_{\Delta t}^n\sphih} -2\elementprod{1}{\brh^{n-1}}{\delta_{\Delta t}^n\brh}  
		+ \elementprod{1}{\delta_{\Delta t}^n\spsih}{\spsih^n} + \elementprod{1}{M(\sphih^{n-1})\bph^n}{\bph^n} 
		&=
		\notag
		\\
		2 \elementprod{1}{\nabla \cdot \brh^{n-1}}{\delta_{\Delta t}^n\sphih}  -2 \elementprod{1}{\sphih^{n-1}} {\nabla \cdot \delta_{\Delta t}^n\brh} + \faceprod{\delta_{\Delta t}^n\fphih}{\bqh^n \cdot \boldsymbol{n}}
		\notag
		\\
		+ 2 \faceprod{\widehat{\brh^{n-1}\cdot \boldsymbol{n}} - \brh^{n-1} \cdot \boldsymbol{n}}{\delta_{\Delta t}^n\sphih} - \faceprod{\widehat{\delta_{\Delta t}^n\brh\cdot \boldsymbol{n}} - \delta_{\Delta t}^n\brh \cdot \boldsymbol{n}}{\spsih^n}
		\notag
		\\
		- \faceprod{\fpsih^n}{\delta_{\Delta t}^n \brh \cdot \boldsymbol{n}} + \faceprod{\widehat{\bqh^n\cdot \boldsymbol{n}} - \bqh^n \cdot \boldsymbol{n}}{\delta_{\Delta t}^n\sphih}  + 2\faceprod{\fphih^{n-1}}{\delta_{\Delta t}^n\brh \cdot \boldsymbol{n}}
		\notag
		\\
		- \faceprod{\fmuh^n}{\bsh^n \cdot \boldsymbol{n}}
		- \faceprod{\widehat{\bsh^n\cdot \boldsymbol{n}} - \bsh^n \cdot \boldsymbol{n}}{\smuh^n}.
		\label{eq:aux_energy_discrete_2}
	\end{align}
	Let us focus on the first two terms on the right hand side of \cref{eq:aux_energy_discrete_2}. We have
	\begin{multline}
		2 \elementprod{1}{\nabla \cdot \brh^{n-1}}{\delta_{\Delta t}^n\sphih}  -2 \elementprod{1}{\sphih^{n-1}} {\nabla \cdot \delta_{\Delta t}^n\brh} = 2 \Delta t^{-1} \elementprod{1}{\nabla \cdot \brh^{n-1}}{\sphih^n} \\
		- 2 \Delta t^{-1} \elementprod{1}{\sphih^{n-1}} {\nabla \cdot \brh^n}. \label{eq:aux_energy_discrete_3}
	\end{multline}
	
	From \cref{eq:r_disc}, taking $\testr = 2\Delta t^{-1}\brh^{n-1}$, we obtain that
	\begin{equation}
		2\Delta t^{-1} \elementprod{1}{\sphih^n} {\nabla \cdot \brh^{n-1}} = 2\Delta t^{-1} \elementprod{1}{\brh^n}{\brh^{n-1}} + 2\Delta t^{-1} \faceprod{\fphih^n}{\brh^{n-1} \cdot \boldsymbol{n}}.
	\end{equation}
	Furthermore, by taking $\testr = 2\Delta t^{-1}\brh^n$ in \cref{eq:r_disc_n1} and using it in the previous equation, we have
	\begin{multline}
		2\Delta t^{-1} \elementprod{1}{\sphih^n} {\nabla \cdot \brh^{n-1}} = 2\Delta t^{-1}\elementprod{1}{\sphih^{n-1}} {\nabla \cdot \brh^n} - 2\Delta t^{-1}\faceprod{\fphih^{n-1}}{\brh^n \cdot \boldsymbol{n}} 
		\\
		+ 2\Delta t^{-1} \faceprod{\fphih^n}{\brh^{n-1} \cdot \boldsymbol{n}}.
	\end{multline}
	Substituting in \cref{eq:aux_energy_discrete_3}, we obtain
	\begin{multline}
		2 \elementprod{1}{\nabla \cdot \brh^{n-1}}{\delta_{\Delta t}^n\sphih}  -2 \elementprod{1}{\sphih^{n-1}} {\nabla \cdot \delta_{\Delta t}^n\brh} = - 2\Delta t^{-1}\faceprod{\fphih^{n-1}}{\brh^n \cdot \boldsymbol{n}} 
		\\
		+ 2\Delta t^{-1} \faceprod{\fphih^n}{\brh^{n-1} \cdot \boldsymbol{n}}.
	\end{multline}
	Using this in \cref{eq:aux_energy_discrete_2}, we have
	\begin{multline}
		\elementprod{1}{f(\sphih^n)}{\delta_{\Delta t}^n\sphih} -2\elementprod{1}{\brh^{n-1}}{\delta_{\Delta t}^n\brh}  
		+ \elementprod{1}{\delta_{\Delta t}^n\spsih}{\spsih^n} + \elementprod{1}{M(\sphih^{n-1})\bph^n}{\bph^n} 
		= 
		\\
		- 2\Delta t^{-1}\faceprod{\fphih^{n-1}}{\brh^n \cdot \boldsymbol{n}} 
		+ 2\Delta t^{-1} \faceprod{\fphih^n}{\brh^{n-1} \cdot \boldsymbol{n}} + \faceprod{\delta_{\Delta t}^n\fphih}{\bqh^n \cdot \boldsymbol{n}}
		\\
		+ 2 \faceprod{\widehat{\brh^{n-1}\cdot \boldsymbol{n}} - \brh^{n-1} \cdot \boldsymbol{n}}{\delta_{\Delta t}^n\sphih} - \faceprod{\widehat{\delta_{\Delta t}^n\brh\cdot \boldsymbol{n}} - \delta_{\Delta t}^n\brh \cdot \boldsymbol{n}}{\spsih^n}
		\\
		- \faceprod{\fpsih^n}{\delta_{\Delta t}^n \brh \cdot \boldsymbol{n}} + \faceprod{\widehat{\bqh^n\cdot \boldsymbol{n}} - \bqh^n \cdot \boldsymbol{n}}{\delta_{\Delta t}^n\sphih}  + 2\faceprod{\fphih^{n-1}}{\delta_{\Delta t}^n\brh \cdot \boldsymbol{n}}
		\\
		- \faceprod{\fmuh^n}{\bsh^n \cdot \boldsymbol{n}}
		- \faceprod{\widehat{\bsh^n\cdot \boldsymbol{n}} - \bsh^n \cdot \boldsymbol{n}}{\smuh^n}  
	\end{multline}
	Substituting the numerical fluxes from \cref{eq:num_fluxes}, and using that $\tau_2 = \tau_1$ and $\tau_4 = 2\tau_1$, we obtain
	\begin{align}
		\elementprod{1}{f(\sphih^n)}{\delta_{\Delta t}^n\sphih} -2\elementprod{1}{\brh^{n-1}}{\delta_{\Delta t}^n\brh}  
		+ \elementprod{1}{\delta_{\Delta t}^n\spsih}{\spsih^n} + \elementprod{1}{M(\sphih^{n-1})\bph^n}{\bph^n} 
		&=
		\notag
		\\
		2 \faceprod{\delta_{\Delta t}^n \fphih}{\brh^{n-1} \cdot \boldsymbol{n}} - \faceprod{\fpsih^n}{\delta_{\Delta t}^n \brh \cdot \boldsymbol{n}}
		+ \faceprod{\delta_{\Delta t}^n\fphih}{\bqh^n \cdot \boldsymbol{n}}
		\notag
		\\
		+ 2 \tau_1 \faceprod{\sphih^{n-1} - \fphih^{n-1}}{\delta_{\Delta t}^n\sphih}  
		-2\tau_1 \faceprod{\sphih^n - \fphih^n}{\delta_{\Delta t}^n\sphih}
		+ \tau_1 \faceprod{\delta_{\Delta t}^n\fphih}{\spsih^n}
		\notag
		\\
		- \tau_1 \faceprod{\fpsih^n}{\delta_{\Delta t}^n\sphih} 
		- \faceprod{\fmuh^n}{\bsh^n \cdot \boldsymbol{n}}
		- \tau_3 \faceprod{\smuh^n - \fmuh^n}{\smuh^n}.
		\label{eq:aux_energy_discrete_4}
	\end{align}
	Let $\testfmu = \fmuh^n$, $\testfpsi = -\delta^n_{\Delta t}\fphi$ in \cref{eq:shat_disc}, and add it to the right hand side of \cref{eq:aux_energy_discrete_4}, to obtain
	\begin{multline}
		\elementprod{1}{f(\sphih^n)}{\delta_{\Delta t}^n\sphih} -2\elementprod{1}{\brh^{n-1}}{\delta_{\Delta t}^n\brh}  
		+ \elementprod{1}{\delta_{\Delta t}^n\spsih}{\spsih^n} + \elementprod{1}{M(\sphih^{n-1})\bph^n}{\bph^n} 
		\\
		+ \tau_3 \faceprod{\smuh^n - \fmuh^n}{\smuh^n - \fmuh^n}
		= 
		2 \faceprod{\delta_{\Delta t}^n \fphih}{\brh^{n-1} \cdot \boldsymbol{n}} - \faceprod{\fpsih^n}{\delta_{\Delta t}^n \brh \cdot \boldsymbol{n}}
		\\
		- \tau_1 \faceprod{\fpsih^n}{\delta_{\Delta t}^n(\sphih - \fphih)} + 2 \tau_1 \faceprod{\sphih^{n-1} - \fphih^{n-1}}{\delta_{\Delta t}^n\sphih}  
		\\
		-2\tau_1 \faceprod{\sphih^n - \fphih^n}{\delta_{\Delta t}^n(\sphih - \fphih)}.
		\label{eq:aux_energy_discrete_5}
	\end{multline}
	Since \cref{eq:rhat_disc} holds for all $n$, we have
	\begin{align}
		\faceprod{\brh^{n-1} \cdot \boldsymbol{n}}{\testfphi} + \tau_1 \faceprod{\sphih^{n-1} - \fphih^{n-1}}{\testfphi} &= 0.
		\label{eq:aux_rhat_n1}
		\\
		\faceprod{\delta_{\Delta t}^n\brh \cdot \boldsymbol{n}}{\testfphi} + \tau_1 \faceprod{\delta_{\Delta t}^n(\sphih - \fphih)}{\testfphi} &= 0
		\label{eq:aux_deltarhat_n1}
	\end{align}
	Take $\testfphi = -2\delta_{\Delta t}^n\fphih$ in \cref{eq:aux_rhat_n1}, and $\testfphi = \fpsih^n$, add up to the right hand side of \cref{eq:aux_energy_discrete_5} and simplify, to get
	\begin{align}
		\elementprod{1}{f(\sphih^n)}{\delta_{\Delta t}^n\sphih} -2\elementprod{1}{\brh^{n-1}}{\delta_{\Delta t}^n\brh}  
		+ \elementprod{1}{\delta_{\Delta t}^n\spsih}{\spsih^n} + \elementprod{1}{M(\sphih^{n-1})\bph^n}{\bph^n} 
		\notag
		\\
		+ \tau_3 \faceprod{\smuh^n - \fmuh^n}{\smuh^n - \fmuh^n}
		+ 2 \tau_1 \Delta t \faceprod{\delta_{\Delta t}^n(\sphih - \fphih)}{\delta_{\Delta t}^n(\sphih - \fphih)}  
		&=0.
		\label{eq:aux_energy_discrete_6}
	\end{align}
	Let us work now on the left hand side of \cref{eq:aux_energy_discrete_2} to see that this is the discrete energy. Using the identities $a(a-b) = \frac 12 (a^2-b^2+(a-b)^2)$ and $b(a-b) = \frac 12 (a^2-b^2-(a-b)^2)$, we have
	\begin{subequations}
		\begin{align}
			\elementprod{1}{\brh^{n-1}}{\delta_{\Delta t}^n\brh} &= \frac 12 \Delta t^{-1}\sum_{K\in \mathcal{T}_h}\int_K\del{\abs[0]{\brh^n}^2 - \abs[0]{\brh^{n-1}}^2 - \abs[0]{\brh^n - \brh^{n-1}}^2}\dif x,
			\label{eq:aux1}
			\\
			\elementprod{1}{\delta_{\Delta t}^n\spsih}{\spsih^n} &= \frac 12 \Delta t^{-1}\sum_{K\in \mathcal{T}_h}\int_K\del{\abs[0]{\spsih^n}^2 - \abs[0]{\spsih^{n-1}}^2 + \abs[0]{\spsih^n - \spsih^{n-1}}^2}\dif x,
			\label{eq:energ_psi_aux}
			\\
			\elementprod{1}{\sphih^n}{\delta_{\Delta t}^n\sphih^n} &= \frac{1}{2} \Delta t^{-1}\sum_{K\in \mathcal{T}_h}\int_K\left(\abs[0]{\sphih^n}^2 - \abs[0]{\sphih^{n-1}}^2+ \abs[0]{\sphih^n - \sphih^{n-1}}^2\right)\dif x.
			\label{eq:aux3}
		\end{align}
	\end{subequations}
	In a similar manner, we can obtain
	\begin{align}
		\elementprod{1}{\del[0]{\sphih^n}^3}{\delta_{\Delta t}^n\sphih}
		= \Delta t^{-1}\sum_{K \in \mathcal{T}_h}\int_K &\left(\frac{1}{4}\del[0]{\sphih^n}^4 - \frac{1}{4}\del[0]{\sphih^{n-1}}^4 + \frac 12 \del[0]{\sphih^n}^2\del[1]{\sphih^n - \sphih^{n-1}}^2\right.
		\notag
		\\
		&\left.+ \frac 14 \del[1]{\sphih^n + \sphih^{n-1}}^2\del[1]{\sphih^n - \sphih^{n-1}}^2\right)\dif x.
		\label{eq:aux2}
	\end{align}
	Using \cref{eq:aux1,eq:energ_psi_aux,eq:aux3} and \cref{eq:aux2} in \cref{eq:aux_energy_discrete_6}
	\begin{multline}
		\Delta t^{-1}\sum_{K\in \mathcal{T}_h}\int_K\del{\frac{1}{4}\del[0]{\sphih^n}^4 + \frac{1-\varepsilon}{2}\abs[0]{\sphih^n}^2 -\abs[0]{\brh^n}^2 + \frac 12 \abs[0]{\spsih^n}^2}\dif x
		+ \norm[1]{M(\sphih^{n-1})^{1/2}\bph^n}^2_{\mathcal{T}_h}
		\\
		+ \Delta t^{-1}\sum_{K\in \mathcal{T}_h}\int_K\del{\frac 12 \del[0]{\sphih^n}^2\del[1]{\sphih^n - \sphih^{n-1}}^2 + \frac 14 \del[1]{\sphih^n + \sphih^{n-1}}^2\del[1]{\sphih^n - \sphih^{n-1}}^2}\dif x
		\\
		+ \tau_3 \faceprod{\smuh^n - \fmuh^n}{\smuh^n - \fmuh^n}
		+ 2 \tau_1 \Delta t \faceprod{\delta_{\Delta t}^n(\sphih - \fphih)}{\delta_{\Delta t}^n(\sphih - \fphih)}  
		\\
		\leq \Delta t^{-1}\sum_{K\in \mathcal{T}_h}\int_K\del{\frac{1}{4}\del[0]{\sphih^{n-1}}^4 + \frac{1-\varepsilon}{2}\abs[0]{\sphih^{n-1}}^2 -\abs[0]{\brh^{n-1}}^2 + \frac 12 \abs[0]{\spsih^{n-1}}^2}\dif x,
	\end{multline}
\end{proof}
\section{Existence of the solution}
\label{sec:existence}

In this section, we show the existence of a unique solution to the HDG discretization \cref{eq:discrete_form}. The proof involves the following steps. We begin by recasting \cref{eq:discrete_form} as a fixed-point problem of a certain mapping $\mathcal{F}$. Then, we define a compact set $X_h$ such that $\mathcal{F}$ maps $X_h$ into itself. This will show the existence of a fixed-point thanks to Brouwer's fixed-point theorem \cite[Theorem 9.9-2]{ciarlet2013linear}:
\begin{theorem}[Brouwer's fixed-point theorem]
	Let $X_h$ be a compact and convex subset of a finite-dimensional normed vector space, and let $\mathcal{F}: X_h \rightarrow X_h$ be a continuous mapping. Then $\mathcal{F}$ has at least one fixed-point in $X_h$.
	\label{thm:brouwers}
\end{theorem}

Hereafter, we assume that $\tau_2 = \tau_1$, as required for energy stability, and $M(\phi) \geq 0$ for all $\phi$.

We will use $a \lesssim b$ to denote that there exists a constant $C>0$ independent of $h$ and $\Delta t$ such that $a \leq Cb$.

\noindent\textbf{Mesh regularity assumption.} For every $K\in \mathcal{T}_h$, there exists a constant $C_1 >0$ such that 
\begin{equation}
	C_1 h_K \leq h_e, \quad \forall e \in \partial K.
	\label{eq:mesh_reg_faces}
\end{equation}
Moreover, there exists a constant $C_2 >0$ such that for every $K\in\mathcal{T}_h$,
\begin{equation}
	C_2 h \leq h_K.
	\label{eq:mesh_reg_domain}
\end{equation}

We will make use of the following norm equivalence \cite[Lemma 1.50]{dipietro_ern}
\begin{lemma}
	Let $\mathcal{T}_h$ be a contact and shape regular mesh sequence. Let $1 \leq p,q \leq \infty$. Then, for all $v_h \in W_h$ and all $K\in\mathcal{T}_h$
	\begin{equation}
		\norm[1]{v_h}_{L^p(K)} \leq C_{inv}^{p,q} h_K^{d(1/p - 1/q)} \norm[1]{v_h}_{L^q(K)},
	\end{equation}
	where $C_{inv}^{p,q}$ only depends on the polynomial degree $k$, the mesh regularity constants $C_1$, $C_2$, and on $p$ and $q$.
	\label{lem:equiv_lp_lq}
\end{lemma}
Moreover, we use that for any $1 \leq p < \infty$, using the Cauchy--Schwarz inequality, we can obtain
\begin{equation}
	\norm[1]{fg}_{L^p(\mathcal{T}_h)}^2 = \del[1]{\elementprod{1}{\abs[0]{f}^p}{\abs[0]{g}^p}}^{2/p} \leq \norm[1]{\abs[0]{f}^p}_{\mathcal{T}_h}^{2/p}\norm[1]{\abs[0]{g}^p}_{\mathcal{T}_h}^{2/p} = \norm[1]{f}_{L^{2p}(\mathcal{T}_h)}^{2}\norm[1]{g}_{L^{2p}(\mathcal{T}_h)}^{2}.
	\label{eq:holders}
\end{equation}

We define the mapping $\mathcal{F}$ as follows. For $u \in W_h$, we take $\mathcal{F}(u)$ to be the component $\sphih^n$ of the solution $(\brh^n, \bqh^n, \bph^n, \bsh^n, \sphih^n, \spsih^n, \smuh^n, \fphih^n, \fpsih^n, \fmuh^n) \in \boldsymbol{V}_h^4 \times W_h^3 \times M_h^3$ of
\begin{subequations}
	\begin{align}
		\elementprod{1}{\brh^n}{\testr} - \elementprod{1}{\sphih^n} {\nabla \cdot \testr} + \faceprod{\fphih^n}{\testr \cdot \boldsymbol{n}} &= 0,
		\label{eq:r_disc_linear}
		\\
		\elementprod{1}{\nabla \cdot \brh^n}{\testphi} + \faceprod{\tau_1(\sphih^n - \fphih^n)}{\testphi} + \elementprod{1}{\spsih^n}{\testphi} &= 0,
		\label{eq:psi_disc_linear}
		\\
		\elementprod{1}{\bqh^n}{\testq} - \elementprod{1}{\spsih^n}{\nabla\cdot \testq} + \faceprod{\fpsih^n}{\testq \cdot \boldsymbol{n}}  &= 0,
		\label{eq:q_disc_linear}
		\\
		\elementprod{1}{\nabla \cdot \bqh^n}{\testpsi} + \faceprod{\tau_2(\spsih^n - \fpsih^n)}{\testpsi} 
		\nonumber
		\\+ \elementprod{1}{\smuh^n}{\testpsi} 
		- (1-\varepsilon)\elementprod{1}{\sphih^n}{\testpsi} - \faceprod{\tau_4(\sphih^n - \fphih^n)}{\testpsi} &= 
		\elementprod{1}{u^3}{\testpsi} 
		\nonumber
		\\
		&+ 2\elementprod{1}{\spsih^{n-1}}{\testpsi} 
		\label{eq:mu_disc_linear}
		\\
		\elementprod{1}{\bph^n}{\testp} - \elementprod{1}{\smuh^n}{\nabla \cdot \testp} + \faceprod{\fmuh^n}{\testp \cdot \boldsymbol{n}} &= 0,
		\label{eq:p_disc_linear}
		\\
		\elementprod{1}{\bsh^n}{\tests} - \elementprod{1}{M(\sphih^{n-1})\bph^n}{\tests} &= 0,
		\label{eq:s_disc_linear}
		\\
		\elementprod{1}{\nabla \cdot \bsh^n}{\testmu} + \faceprod{\tau_3(\smuh^n - \fmuh^n)}{\testmu} + \Delta t^{-1}\elementprod{1}{\sphih^n}{\testmu} &= \Delta t^{-1}\elementprod{1}{\sphih^{n-1}}{\testmu},
		\label{eq:phi_disc_linear}
		\\
		\faceprod{\brh^n \cdot \boldsymbol{n} + \tau_1(\sphih^n - \fphih^n)}{\testfphi} &= 0,
		\label{eq:rhat_disc_linear}
		\\
		\faceprod{\bqh^n \cdot \boldsymbol{n} + \tau_2(\spsih^n - \fpsih^n) - \tau_4(\sphih^n - \fphih^n)}{\testfpsi} &= 0,
		\label{eq:qhat_disc_linear}
		\\
		\faceprod{\bsh^n \cdot \boldsymbol{n} + \tau_3(\smuh^n - \fmuh^n)}{\testfmu} &= 0.
		\label{eq:shat_disc_linear}
	\end{align}
	\label{eq:linear_discrete_form}
\end{subequations}
We note that a fixed-point of $\mathcal{F}$ is a solution to \cref{eq:discrete_form}. Moreover, for a given $u \in W_h$, \cref{eq:linear_discrete_form} is a linear problem. First, note that by taking $\testr = -\bqh^n,\,\testphi = \spsih^n,\,\testq = \brh^n,\,\testpsi = -\sphih^n, \,\testp=\Delta t \bsh^n,\,\tests=-\Delta t \bph^n, \,\testmu=\Delta t \smuh^n, \,\testfphi = -\fpsih^n, \,\testfpsi = \fphih^n, \,\testfmu = -\Delta t \fmuh^n$ in \cref{eq:linear_discrete_form}, and adding up the equations, we obtain the following relation
\begin{align}
	\norm[1]{\spsih^n}^2_{\mathcal{T}_h} + (1-\varepsilon)\norm[1]{\sphih^n}^2_{\mathcal{T}_h} + \Delta t\norm[1]{M^{1/2}(\sphih^{n-1}) \bph^n}^2_{\mathcal{T}_h} + \tau_3 \Delta t \norm[1]{\smuh^n - \fmuh^n}^2_{\partial\mathcal{T}_h} 
	\nonumber
	\\
	+ \tau_4 \norm[1]{\sphih^n - \fphih^n}^2_{\partial\mathcal{T}_h} = -\elementprod{1}{u^3}{\sphih^n} - 2\elementprod{1}{\spsih^{n-1}}{\sphih^n} + \elementprod{1}{\sphih^{n-1}}{\smuh^n}.
	\label{eq:relation_linear}
\end{align}

We will now show that $\mathcal{F}$ is a well-defined continuous mapping on
\begin{equation}
	X_h = \cbr[1]{\sphih \in W_h : \norm[1]{\sphih}^2_{\mathcal{T}_h} \leq R},
	\label{eq:Xh_space}
\end{equation}
with
$$
R = \frac{1-\varepsilon}{3\del[1]{C_{inv}^{6,2}}^{3}\sqrt{5 - \varepsilon}}\sqrt{3} \,h^{d}.
$$

\begin{lemma}
	For a given mesh of size $h$, the mapping $\mathcal{F}$ is well-defined and continuous on the compact and convex subset $X_h$ defined in \cref{eq:Xh_space}.
	\label{lem:wellposed_cont}
\end{lemma}
\begin{proof}
	
	\noindent\textbf{Well-posedness.}
	Since \cref{eq:linear_discrete_form} is a linear problem, we take all source terms equal to zero and show that the only solution is the trivial one. In this case, the relation \cref{eq:relation_linear} gives
	\begin{align}
		\norm[1]{\spsih^n}^2_{\mathcal{T}_h} + (1-\varepsilon)\norm[1]{\sphih^n}^2_{\mathcal{T}_h} + \Delta t \norm[1]{M^{1/2}(\sphih^{n-1}) \bph^n}^2_{\mathcal{T}_h} &+ \tau_3 \Delta t \norm[1]{\smuh^n - \fmuh^n}^2_{\partial\mathcal{T}_h} 
		\nonumber
		\\
		&+ \tau_4 \norm[1]{\sphih^n - \fphih^n}^2_{\partial\mathcal{T}_h} = 0.
	\end{align}
	Thus, $M(\sphih^{n-1})\bph^n = \boldsymbol{0}$ in $\Omega$, $\sphih^n = \spsih^n = 0$ in $\Omega$, and $\smuh^n=\fmuh^n$, $\fphih^n = 0$ on $\partial K$, for all $K\in\mathcal{T}_h$. We take $\tests = \bsh^n$ in \cref{eq:s_disc_linear} and obtain that $\bsh^n = \boldsymbol{0}$ in $\Omega$.
	Taking $\testq = \bqh^n$ and $\testfpsi = -\fpsih^n$ in \cref{eq:q_disc_linear} and \cref{eq:qhat_disc_linear} respectively, and adding up, we obtain
	\begin{equation}
		\norm[1]{\bqh^n}^2_{\mathcal{T}_h} + \tau_2\norm[1]{\fpsih^n}^2_{\partial\mathcal{T}_h} = 0,
	\end{equation}
	from where we conclude that $\bqh^n = \boldsymbol{0}$ and $\fpsih^n = 0$. Then, taking $\testpsi = \smuh^n$ in \cref{eq:mu_disc_linear}, we conclude that $\smuh^n = 0$, and therefore, $\fmuh^n = 0$. Taking $\testr = \brh^n$ in \cref{eq:r_disc_linear} we conclude that $\brh^n = \boldsymbol{0}$. Finally, taking $\testp = \bph^n$ in \cref{eq:p_disc_linear}, and since $\smuh^n = 0$ and $\fmuh^n = 0$, we obtain $\bph^n = \boldsymbol{0}$.
	
	\noindent\textbf{Continuity.}
	We consider a sequence $\cbr[0]{u_m}_{m=1}^{\infty}$, where $u_m \in X_h, \forall m$ such that $\displaystyle \lim_{m\rightarrow \infty} u_m = u_\ast$. Since $X_h$ is compact, then $u_\ast \in X_h$. Let $(\br_m, \bq_m, \bp_m, \bs_m, \phi_m, \allowbreak  \psi_m, \mu_m, \bar{\phi}_m,\bar{\psi}_m, \bar{\mu}_m)$ be the solution of \cref{eq:linear_discrete_form} corresponding to $u_m$, and similarly, $(\br_\ast, \bq_\ast, \bp_\ast, \bs_\ast,\phi_\ast, \allowbreak \psi_\ast, \mu_\ast, \bar{\phi}_\ast, \bar{\psi}_\ast, \bar{\mu}_\ast)$ be the solution of \cref{eq:linear_discrete_form} corresponding to $u_\ast$. We denote the difference by $(\br^{m,\ast}, \bq^{m,\ast}, \bp^{m,\ast}, \bs^{m,\ast}, \phi^{m,\ast}, \allowbreak \psi^{m,\ast}, \mu^{m,\ast}, \bar{\phi}^{m,\ast}, \bar{\psi}^{m,\ast}, \bar{\mu}^{m,\ast})$, i.e., $\br^{m,\ast} = \br_m - \br_\ast$, and similar for the other variables. We want to show that $\displaystyle \lim_{m\rightarrow \infty} \phi_m = \phi_\ast$. The relation in \cref{eq:relation_linear} for the difference is
	\begin{align}
		&\norm[1]{\psi^{m,\ast}}_{\mathcal{T}_h}^2
		+ (1-\varepsilon)\norm[1]{\phi^{m,\ast}}^2_{\mathcal{T}_h} 
		+ \Delta t\norm[1]{M^{1/2}(\sphih^{n-1})\bp^{m,\ast}}_{\mathcal{T}_h}^2 \nonumber
		\\
		&+ \tau_3\Delta t\norm[1]{\smu^{m,\ast} - \fmu^{m,\ast}}_{\partial\mathcal{T}_h}^2
		+ \tau_4\norm[1]{\sphi^{m,\ast} - \fphi^{m,\ast}}_{\partial\mathcal{T}_h}^2 = -\elementprod{1}{u_m^3 - (u_\ast)^3}{\phi^{m,\ast}}.
		\label{eq:aux_lim2}
	\end{align}
	Using Cauchy--Schwarz and Young's inequalities on the right hand side, we obtain
	\begin{equation}
		\norm[1]{\phi_m - \phi_\ast}^2_{\mathcal{T}_h} = \norm[1]{\phi^{m,\ast}}^2_{\mathcal{T}_h} \lesssim \norm[1]{u_m^3 - u_\ast^3}^2_{\mathcal{T}_h}.
	\end{equation}
	Using the identity $a^3 - b^3 = (a-b)(a^2+ab+b^2)$, the triangle inequality, and the inequality in \cref{eq:holders} on the right hand side, we have
	\begin{align}
		\norm[1]{\phi_m - \phi_\ast}^2_{\mathcal{T}_h} 
		& \lesssim \norm[2]{\del[1]{u_m - u_\ast}\del[1]{u_m^2 + u_mu_\ast + u_\ast}}_{\mathcal{T}_h}^2
		\notag
		\\
		& \leq \norm[1]{u_m - u_\ast}_{L^4(\mathcal{T}_h)}^2 \norm[1]{u_m^2 + u_mu_\ast + u_\ast}_{L^4(\mathcal{T}_h)}^2
		\notag
		\\
		& \lesssim \norm[1]{u_m - u_\ast}^2_{L^4(\mathcal{T}_h)}\del{\norm[1]{u_m}_{L^8(\mathcal{T}_h)}^4 + \norm[1]{u_mu_\ast}^2_{L^4(\mathcal{T}_h)} + \norm[1]{u_\ast}^4_{L^8(\mathcal{T}_h)}}
		\notag
		\\
		& \leq \norm[1]{u_m - u_\ast}^2_{L^4(\mathcal{T}_h)}\del{\norm[1]{u_m}^4_{L^8(\mathcal{T}_h)} + \norm[1]{u_m}^2_{L^8(\mathcal{T}_h)}\norm[1]{u_\ast}^2_{L^8(\mathcal{T}_h)} + \norm[1]{u_\ast}^4_{L^8(\mathcal{T}_h)}}.
	\end{align}
	Moreover, using the inequality in \cref{lem:equiv_lp_lq}, we obtain
	\begin{align}
		\norm[1]{\phi_m - \phi_\ast}^2_{\mathcal{T}_h} & \lesssim h^{-3d/2} \norm[1]{u_m - u_\ast}^2_{L^4(\mathcal{T}_h)}\del{\norm[1]{u_m}^4_{\mathcal{T}_h} + \norm[1]{u_m}^2_{\mathcal{T}_h}\norm[1]{u_\ast}^2_{\mathcal{T}_h} + \norm[1]{u_\ast}^4_{\mathcal{T}_h}}.
	\end{align}
	Since $u_m, \, u_\ast \in X_h$, we have
	\begin{align}
		\norm[1]{\phi_m - \phi_\ast}^2_{\mathcal{T}_h} & \lesssim \norm[1]{u_m - u_\ast}^2_{\mathcal{T}_h}h^{-3d/2}R^4
		\notag
		\\
		&= \norm[1]{u_m - u_\ast}^2_{\mathcal{T}_h}\del[0]{12C_{inv}^{6,2}}^{-2}h^{5d/2}.
	\end{align}
	Thus, if $u_m \rightarrow u_\ast$, then $\phi_m \rightarrow \phi_\ast$, and we conclude that $\mathcal{F}$ is continuous on $X_h$.
\end{proof}

We are now ready to show the existence of a solution.

\begin{theorem}[Existence of a solution]
	Assume that $\spsih^{n-1}$ and $\sphih^{n-1}$ are small enough, i.e.,
	\begin{align}
		\frac{7-4\varepsilon}{1-\varepsilon} \norm[1]{\sphih^{n-1}}^2_{\mathcal{T}_h} 
		&+ \frac{11-3\varepsilon}{(1-\varepsilon)^2}\norm[1]{\spsih^{n-1}}^2_{\mathcal{T}_h}
		\notag
		\\
		&+ \frac{\tau_4}{1-\varepsilon}\norm[1]{\sphih^{n-1} - \fphih^{n-1}}^2_{\partial\mathcal{T}_h} \leq \frac{1-\varepsilon}{9\del[1]{C_{inv}^{6,2}}^3\sqrt{5-\varepsilon}} 2\sqrt{3}\,h^d.
		\label{eq:small_data_cond}
	\end{align}
	Then, the discrete problem in \cref{eq:discrete_form} has at least one solution, if $\tau_1,\tau_2,\tau_3,\tau_4 >0$ and $\tau_2 = \tau_1$.
\end{theorem}

\begin{proof}
	We will show that if $u \in X_h$, then $\mathcal{F}(u) \in X_h$. 
	From \cref{eq:relation_linear}, we have
	\begin{align}
		\norm[1]{\spsih^n}^2_{\mathcal{T}_h} + (1-\varepsilon)\norm[1]{\sphih^n}^2_{\mathcal{T}_h} + \Delta t\norm[1]{M^{1/2}(\sphih^{n-1}) \bph^n}^2_{\mathcal{T}_h} + \tau_3 \Delta t \norm[1]{\smuh^n - \fmuh^n}^2_{\partial\mathcal{T}_h} 
		\nonumber
		\\
		+ \tau_4 \norm[1]{\sphih^n - \fphih^n}^2_{\partial\mathcal{T}_h} = -\elementprod{1}{u^3}{\sphih^n} - 2\elementprod{1}{\spsih^{n-1}}{\sphih^n} + \elementprod{1}{\sphih^{n-1}}{\smuh^n}.
	\end{align}
	Moreover, we take $\testpsi = \sphih^{n-1}$ in \cref{eq:mu_disc_linear}, $\testq = -\brh^{n-1}$ in \cref{eq:q_disc_linear} and $\testfpsi = \fphih^{n-1}$ in \cref{eq:qhat_disc_linear}. Furthermore, since \cref{eq:linear_discrete_form} holds for $n-1$ as well, we take $\testr = \bqh^n$, $\testphi = -\spsih^n$ and $\testfphi = \fpsih^n$ and add up to obtain
	\begin{multline}
		\norm[1]{\spsih^n}^2_{\mathcal{T}_h} +
		(1-\varepsilon)\norm[1]{\sphih^n}^2_{\mathcal{T}_h} + \Delta t\norm[1]{M^{1/2}(\sphih^{n-1}) \bph^n}^2_{\mathcal{T}_h} + \tau_3 \Delta t \norm[1]{\smuh^n - \fmuh^n}^2_{\partial\mathcal{T}_h} 
		\\
		+ \tau_4 \norm[1]{\sphih^n - \fphih^n}^2_{\partial\mathcal{T}_h} = -\elementprod{1}{u^3}{\sphih^n} + \elementprod{1}{u^3}{\sphih^{n-1}} - 2\elementprod{1}{\spsih^{n-1}}{\sphih^n} + 2\elementprod{1}{\spsih^{n-1}}{\sphih^{n-1}} 
		\\
		+ \elementprod{1}{\spsih^{n-1}}{\spsih^n}
		+ (1-\varepsilon)\elementprod{1}{\sphih^{n}}{\sphih^{n-1}} + \tau_4\faceprod{\sphih^n - \fphih^n}{\sphih^{n-1} - \fphih^{n-1}}.
		\label{eq:exist1}
	\end{multline}
	Using Cauchy--Schwarz and Young's inequalities on the right hand side of \cref{eq:exist1}, we have
	\begin{multline}
		\norm[1]{\spsih^n}^2_{\mathcal{T}_h} + (1-\varepsilon)\norm[1]{\sphih^n}^2_{\mathcal{T}_h} + \Delta t\norm[1]{M^{1/2}(\sphih^{n-1}) \bph^n}^2_{\mathcal{T}_h} + \tau_3 \Delta t \norm[1]{\smuh^n - \fmuh^n}^2_{\partial\mathcal{T}_h} 
		\\
		+ \tau_4 \norm[1]{\sphih^n - \fphih^n}^2_{\partial\mathcal{T}_h} \leq \frac{1}{2\varepsilon_1}\norm[1]{u^3}^2_{\mathcal{T}_h} + \frac{\varepsilon_1}{2}\norm[1]{\sphih^n}^2_{\mathcal{T}_h} + \frac{1}{2\varepsilon_2}\norm[1]{u^3}^2_{\mathcal{T}_h} + \frac{\varepsilon_2}{2}\norm[1]{\sphih^{n-1}}^2_{\mathcal{T}_h} 
		\\
		+ \frac{1}{\varepsilon_3}\norm[1]{\spsih^{n-1}}^2_{\mathcal{T}_h} + \varepsilon_3\norm[1]{\sphih^n}^2_{\mathcal{T}_h} + \frac{1}{\varepsilon_4}\norm[1]{\spsih^{n-1}}^2_{\mathcal{T}_h} + \varepsilon_4\norm[1]{\sphih^{n-1}}^2_{\mathcal{T}_h} 
		+ \frac{1-\varepsilon}{2\varepsilon_5}\norm[1]{\sphih^{n-1}}^2_{\mathcal{T}_h} 
		\\
		+ (1-\varepsilon)\frac{\varepsilon_5}{2}\norm[1]{\sphih^n}^2_{\mathcal{T}_h} + \frac{\tau_4}{2\varepsilon_6}\norm[1]{\sphih^{n-1} - \fphih^{n-1}}^2_{\partial\mathcal{T}_h} + \tau_4\frac{\varepsilon_6}{2}\norm[1]{\sphih^{n} - \fphih^{n}}^2_{\partial\mathcal{T}_h}
		\\
		+ \frac{1}{2\varepsilon_7}\norm[1]{\spsih^{n-1}}^2_{\mathcal{T}_h} + \frac{\varepsilon_7}{2}\norm[1]{\spsih^n}^2_{\mathcal{T}_h}.
	\end{multline}
	Taking $\varepsilon_1 = \varepsilon_3 = \varepsilon_5 = \frac{1-\varepsilon}{4-\varepsilon}$, $\varepsilon_2 = \varepsilon_4 = 1-\varepsilon$, $\varepsilon_6 = \varepsilon_7 = 1$, and rearranging terms, we obtain
	\begin{multline}
		\frac{1}{2}\norm[1]{\spsih^n}^2_{\mathcal{T}_h} + \frac{1-\varepsilon}{2}\norm[1]{\sphih^n}^2_{\mathcal{T}_h} + \Delta t\norm[1]{M^{1/2}(\sphih^{n-1}) \bph^n}^2_{\mathcal{T}_h} + \tau_3 \Delta t \norm[1]{\smuh^n - \fmuh^n}^2_{\partial\mathcal{T}_h} 
		\\
		+ \frac{\tau_4}{2}\norm[1]{\sphih^n - \fphih^n}^2_{\partial\mathcal{T}_h} \leq \frac{4-\varepsilon}{2(1-\varepsilon)}\norm[1]{u^3}^2_{\mathcal{T}_h} + \frac{1}{2(1-\varepsilon)}\norm[1]{u^3}^2_{\mathcal{T}_h} + \frac{1-\varepsilon}{2}\norm[1]{\sphih^{n-1}}^2_{\mathcal{T}_h} 
		\\
		+ \frac{4-\varepsilon}{1-\varepsilon}\norm[1]{\spsih^{n-1}}^2_{\mathcal{T}_h} + \frac{1}{1-\varepsilon}\norm[1]{\spsih^{n-1}}^2_{\mathcal{T}_h} + (1-\varepsilon)\norm[1]{\sphih^{n-1}}^2_{\mathcal{T}_h} 
		\\
		+ \frac{4-\varepsilon}{2}\norm[1]{\sphih^{n-1}}^2_{\mathcal{T}_h} + \frac{\tau_4}{2}\norm[1]{\sphih^{n-1} - \fphih^{n-1}}^2_{\partial\mathcal{T}_h} + \frac{1}{2}\norm[1]{\spsih^{n-1}}^2_{\mathcal{T}_h}.
	\end{multline}
	That is,
	\begin{multline}
		\norm[1]{\sphih^n}^2_{\mathcal{T}_h} \leq \frac{5-\varepsilon}{(1-\varepsilon)^2}\norm[1]{u}^6_{L^6(\mathcal{T}_h)} + \frac{7-4\varepsilon}{1-\varepsilon} \norm[1]{\sphih^{n-1}}^2_{\mathcal{T}_h} 
		\\
		+ \frac{11-3\varepsilon}{(1-\varepsilon)^2}\norm[1]{\spsih^{n-1}}^2_{\mathcal{T}_h}
		+ \frac{\tau_4}{1-\varepsilon}\norm[1]{\sphih^{n-1} - \fphih^{n-1}}^2_{\partial\mathcal{T}_h}.
	\end{multline}
	Using \cref{lem:equiv_lp_lq} on the first term of the right hand side, and $\norm[1]{u}^2_{\mathcal{T}_h} \leq R$, we have
	\begin{multline}
		\norm[1]{\sphih^n}^2_{\mathcal{T}_h} \leq \frac{5-\varepsilon}{(1-\varepsilon)^2}\del[1]{C_{inv}^{6,2}}^6h^{-2d}R^3 + \frac{7-4\varepsilon}{1-\varepsilon} \norm[1]{\sphih^{n-1}}^2_{\mathcal{T}_h} 
		\\
		+ \frac{11-3\varepsilon}{(1-\varepsilon)^2}\norm[1]{\spsih^{n-1}}^2_{\mathcal{T}_h}
		+ \frac{\tau_4}{1-\varepsilon}\norm[1]{\sphih^{n-1} - \fphih^{n-1}}^2_{\partial\mathcal{T}_h}.
	\end{multline}
	If the small data assumption in \cref{eq:small_data_cond} is satisfied, then we have
	\begin{align}
		\norm[1]{\sphih^n}^2_{\mathcal{T}_h} &\leq \frac{5-\varepsilon}{(1-\varepsilon)^2}\del[1]{C_{inv}^{6,2}}^6h^{-2d}R^3 + \frac{1-\varepsilon}{9\del[1]{C_{inv}^{6,2}}^3\sqrt{5-\varepsilon}} 2\sqrt{3}\,h^d
		\nonumber
		\\
		&= \frac{5-\varepsilon}{(1-\varepsilon)^2}\del[1]{C_{inv}^{6,2}}^6h^{-2d}\del{\frac{1-\varepsilon}{3\del[1]{C_{inv}^{6,2}}^{3}\sqrt{5 - \varepsilon}}\sqrt{3} \,h^{d}}^3 
		\nonumber
		\\
		&\quad\quad\quad+ \frac{1-\varepsilon}{9\del[1]{C_{inv}^{6,2}}^3\sqrt{5-\varepsilon}} 2\sqrt{3}\,h^d
		\nonumber
		\\
		&= \frac{1-\varepsilon}{3\del[1]{C_{inv}^{6,2}}^{3}\sqrt{5 - \varepsilon}}\sqrt{3} \,h^{d} = R.
	\end{align}
	By Brouwer's fixed-point \cref{thm:brouwers}, $\mathcal{F}$ has at least one fixed-point in $X_h$, and therefore, \cref{eq:discrete_form} has at least one solution.
\end{proof}
Before showing the uniqueness of the solution, we prove that any solution satisfies a bound.
\begin{corollary}\label{corollary-4.5}
	Any fixed-point of $\mathcal{F}$ satisfies the following bound
	\begin{equation}
		\norm[1]{\sphih^n}^2_{\mathcal{T}_h} \leq \frac{4}{1-\varepsilon}B(\sphih^{n-1},\fphih^{n-1}, \spsih^{n-1}),
		\label{eq:bound_sol}
	\end{equation}
	where
	\begin{align}
		B(\sphih^{n-1},\fphih^{n-1}, \spsih^{n-1}) :=& C_1\norm[1]{\sphih^{n-1}}_{\mathcal{T}_h}^2 + C_2\norm[1]{\sphih^{n-1}}^4_{L^4(\mathcal{T}_h)} 
		\nonumber
		\\
		&+C_3\norm[1]{\spsih^{n-1}}_{\mathcal{T}_h}^2+\frac{\tau_4}{2}\norm[1]{\sphih^{n-1} - \fphih^{n-1}}^2_{\partial\mathcal{T}_h},
	\end{align}
	with $C_1 = 4+\frac{1-\varepsilon}{2}$, $C_2 = \frac{1}{4}$ and $C_3 = \frac{1}{2}+\frac{4}{1-\varepsilon}$.
\end{corollary}
\begin{proof}
	Take $\testpsi = -\delta_{\Delta t}^n\sphih$, $\testmu = \smuh^n$, $\testp = \bsh^n$, $\tests = -\bph^n$ in \cref{eq:discrete_form}, add up and simplify to obtain
	\begin{align}
		&\elementprod{1}{f(\sphih^n)}{\delta^n_{\Delta t}\sphih} + \norm[1]{M^{1/2}(\sphih^{n-1})\bph^n}^2_{\mathcal{T}_h} = -2\elementprod{1}{\spsih^{n-1}}{\delta^n_{\Delta t}\sphih} \nonumber
		\\
		&+ \elementprod{1}{\nabla\cdot\bqh^n}{\delta^n_{\Delta t}\sphih} 
		+ \faceprod{\tau_2(\spsih^n - \fpsih^n)}{\delta^n_{\Delta t}\sphih} - \faceprod{\tau_4(\sphih^n - \fphih^n)}{\delta^n_{\Delta t}\sphih} 
		\nonumber
		\\
		&- \faceprod{\tau_3(\smuh^n-\fmuh^n)}{\smuh^n} - \faceprod{\fmuh^n}{\bsh^n\cdot\boldsymbol{n}},
		\label{eq:bounds_aux1}
	\end{align}
	where we have used that $M(\phi) \geq 0$.
	Taking $\testr = -\bqh^n$ in \cref{eq:deltar_disc_n1} and adding it to \cref{eq:bounds_aux1}, we obtain
	\begin{align}
		&\elementprod{1}{f(\sphih^n)}{\delta^n_{\Delta t}\sphih} + \norm[1]{M^{1/2}(\sphih^{n-1})\bph^n}^2_{\mathcal{T}_h} - \elementprod{1}{\delta^n_{\Delta t}\brh}{\bqh^n}= 
		\nonumber
		\\
		&-2\elementprod{1}{\spsih^{n-1}}{\delta^n_{\Delta t}\sphih} 
		+ \faceprod{\delta^n_{\Delta t}\fphih}{\bqh^n\cdot\boldsymbol{n}} 
		+ \faceprod{\tau_2(\spsih^n - \fpsih^n)}{\delta^n_{\Delta t}\sphih} \nonumber
		\\
		&- \faceprod{\tau_4(\sphih^n - \fphih^n)}{\delta^n_{\Delta t}\sphih} 
		- \faceprod{\tau_3(\smuh^n-\fmuh^n)}{\smuh^n} - \faceprod{\fmuh^n}{\bsh^n\cdot\boldsymbol{n}}.
		\label{eq:bounds_aux2}
	\end{align}
	From \cref{eq:q_disc} we can obtain
	\begin{equation}
		-\elementprod{1}{\bqh^n}{\delta^n_{\Delta t}\brh} = - \elementprod{1}{\spsih^n}{\nabla\cdot\delta^n_{\Delta t}\brh} + \faceprod{\fpsih^n}{\delta^n_{\Delta t}\brh\cdot\boldsymbol{n}}.
		\label{eq:bounds_auxq1}
	\end{equation}
	Moreover, from \cref{eq:psi_disc} we have
	\begin{equation}
		-\elementprod{1}{\nabla\cdot\delta^n_{\Delta t}\brh}{\spsih^n} =  \elementprod{1}{\delta^n_{\Delta t}\spsih}{\spsih^n} + \faceprod{\tau_1(\delta^n_{\Delta t}\sphih - \delta^n_{\Delta t}\fphih)}{\spsih^n}.
	\end{equation}
	Using this in the right hand side of \cref{eq:bounds_auxq1} and substituting in  \cref{eq:bounds_aux2}, we have
	\begin{align}
		&\elementprod{1}{f(\sphih^n)}{\delta^n_{\Delta t}\sphih} + \norm[1]{M^{1/2}(\sphih^{n-1})\bph^n}^2_{\mathcal{T}_h} + \elementprod{1}{\delta^n_{\Delta t}\spsih}{\spsih^n} 
		\nonumber
		\\
		&= -2\elementprod{1}{\spsih^{n-1}}{\delta^n_{\Delta t}\sphih} 
		+ \faceprod{\delta^n_{\Delta t}\fphih}{\bqh^n\cdot\boldsymbol{n}} 
		+ \faceprod{\tau_2(\spsih^n - \fpsih^n)}{\delta^n_{\Delta t}\sphih} \nonumber
		\\
		&- \faceprod{\tau_4(\sphih^n - \fphih^n)}{\delta^n_{\Delta t}\sphih} - \faceprod{\tau_3(\smuh^n-\fmuh^n)}{\smuh^n} - \faceprod{\fmuh^n}{\bsh^n\cdot\boldsymbol{n}}
		\nonumber
		\\
		&- \faceprod{\tau_1(\delta^n_{\Delta t}\sphih - \delta^n_{\Delta t}\fphih)}{\spsih^n} - \faceprod{\fpsih^n}{\delta^n_{\Delta t}\brh\cdot\boldsymbol{n}}.
		\label{eq:bounds_aux3}
	\end{align}
	Taking $\testfmu = \fmuh^n$ in \cref{eq:shat_disc}, and using this in \cref{eq:bounds_aux3}, we have
	\begin{align}
		&\elementprod{1}{f(\sphih^n)}{\delta^n_{\Delta t}\sphih} + \norm[1]{M^{1/2}(\sphih^{n-1})\bph^n}^2_{\mathcal{T}_h} + \elementprod{1}{\delta^n_{\Delta t}\spsih}{\spsih^n} 
		\nonumber
		\\
		&+ \faceprod{\tau_3(\smuh^n - \fmuh^n)}{\smuh^n - \fmuh^n}
		= -2\elementprod{1}{\spsih^{n-1}}{\delta^n_{\Delta t}\sphih} + \faceprod{\delta^n_{\Delta t}\fphih}{\bqh^n\cdot\boldsymbol{n}} 
		\nonumber
		\\
		&+ \faceprod{\tau_2(\spsih^n - \fpsih^n)}{\delta^n_{\Delta t}\sphih} 
		- \faceprod{\tau_4(\sphih^n - \fphih^n)}{\delta^n_{\Delta t}\sphih} \nonumber
		\\
		&- \faceprod{\tau_1(\delta^n_{\Delta t}\sphih - \delta^n_{\Delta t}\fphih)}{\spsih^n} 
		- \faceprod{\fpsih^n}{\delta^n_{\Delta t}\brh\cdot\boldsymbol{n}}.
		\label{eq:bounds_aux4}
	\end{align}
	From \cref{eq:rhat_disc}, we obtain
	\begin{equation}
		-\faceprod{\delta^n_{\Delta t}\brh\cdot\boldsymbol{n}}{\fpsih^n} = \faceprod{\tau_1(\delta^n_{\Delta t}\sphih - \delta^n_{\Delta t}\fphih)}{\fpsih^n}.
	\end{equation}
	Therefore, \cref{eq:bounds_aux4} becomes
	\begin{align}
		&\elementprod{1}{f(\sphih^n)}{\delta^n_{\Delta t}\sphih} + \norm[1]{M^{1/2}(\sphih^{n-1})\bph^n}^2_{\mathcal{T}_h} + \elementprod{1}{\delta^n_{\Delta t}\spsih}{\spsih^n} 
		\nonumber
		\\
		&+ \faceprod{\tau_3(\smuh^n - \fmuh^n)}{\smuh^n - \fmuh^n}
		= -2\elementprod{1}{\spsih^{n-1}}{\delta^n_{\Delta t}\sphih} + \faceprod{\delta^n_{\Delta t}\fphih}{\bqh^n\cdot\boldsymbol{n}} 
		\nonumber
		\\
		&+ \faceprod{\tau_2(\spsih^n - \fpsih^n)}{\delta^n_{\Delta t}\sphih} 
		- \faceprod{\tau_4(\sphih^n - \fphih^n)}{\delta^n_{\Delta t}\sphih} \nonumber
		\\
		&- \faceprod{\tau_1(\delta^n_{\Delta t}\sphih - \delta^n_{\Delta t}\fphih)}{\spsih^n - \fpsih^n}.
	\end{align}
	Since $\tau_2 = \tau_1$, we have
	\begin{align}
		&\elementprod{1}{f(\sphih^n)}{\delta^n_{\Delta t}\sphih} + \norm[1]{M^{1/2}(\sphih^{n-1})\bph^n}^2_{\mathcal{T}_h} + \elementprod{1}{\delta^n_{\Delta t}\spsih}{\spsih^n} 
		\nonumber
		\\
		&+ \faceprod{\tau_3(\smuh^n - \fmuh^n)}{\smuh^n - \fmuh^n}
		= -2\elementprod{1}{\spsih^{n-1}}{\delta^n_{\Delta t}\sphih} + \faceprod{\delta^n_{\Delta t}\fphih}{\bqh^n\cdot\boldsymbol{n} \nonumber
			\\
			&+ \tau_2(\spsih^n - \fpsih^n)} 
		- \faceprod{\tau_4(\sphih^n - \fphih^n)}{\delta^n_{\Delta t}\sphih}.
		\label{eq:bounds_aux5}
	\end{align}
	From \cref{eq:qhat_disc} taking $\testfpsi = \delta_{\Delta t}^n\fphih$, we have
	\begin{equation}
		\faceprod{\bqh^n\cdot \boldsymbol{n} + \tau_2(\spsih^n - \fpsih^n)}{\delta_{\Delta t}^n\fphih} = \faceprod{\tau_4(\sphih^n - \fphih^n)}{\delta_{\Delta t}^n\fphih}.
	\end{equation}
	Thus, \cref{eq:bounds_aux5} becomes
	\begin{align}
		&\elementprod{1}{f(\sphih^n)}{\delta^n_{\Delta t}\sphih} + \norm[1]{M^{1/2}(\sphih^{n-1})\bph^n}^2_{\mathcal{T}_h} + \elementprod{1}{\delta^n_{\Delta t}\spsih}{\spsih^n} 
		\nonumber
		\\
		&+ \faceprod{\tau_3(\smuh^n - \fmuh^n)}{\smuh^n - \fmuh^n}
		+ \faceprod{\tau_4(\sphih^n - \fphih^n)}{\delta^n_{\Delta t}\sphih - \delta^n_{\Delta t}\fphih} =
		\nonumber
		\\
		&-2\elementprod{1}{\spsih^{n-1}}{\delta^n_{\Delta t}\sphih}.
		\label{eq:bounds_aux6}
	\end{align}
	Finally, note that
	\begin{multline}
		\faceprod{\tau_4(\sphih^n - \fphih^n)}{\delta^n_{\Delta t}\sphih - \delta^n_{\Delta t}\fphih} = \Delta t^{-1}\faceprod{\tau_4(\sphih^n - \fphih^n)}{\sphih^n - \fphih^n} 
		\nonumber
		\\
		- \Delta t^{-1}\faceprod{\tau_4(\sphih^n - \fphih^n)}{\sphih^{n-1} - \fphih^{n-1}}.
	\end{multline}
	Therefore, 
	\begin{align}
		&\elementprod{1}{f(\sphih^n)}{\delta^n_{\Delta t}\sphih} + \norm[1]{M^{1/2}(\sphih^{n-1})\bph^n}^2_{\mathcal{T}_h} + \elementprod{1}{\delta^n_{\Delta t}\spsih}{\spsih^n} 
		\nonumber
		\\
		&+ \faceprod{\tau_3(\smuh^n - \fmuh^n)}{\smuh^n - \fmuh^n}
		+ \Delta t^{-1}\faceprod{\tau_4(\sphih^n - \fphih^n)}{\sphih^n - \fphih^n} = \nonumber
		\\
		&-2\elementprod{1}{\spsih^{n-1}}{\delta^n_{\Delta t}\sphih} 
		+ \Delta t^{-1}\faceprod{\tau_4(\sphih^n - \fphih^n)}{\sphih^{n-1} - \fphih^{n-1}}.
		\label{eq:bounds_final1}
	\end{align}
	Using \cref{eq:energ_psi_aux,eq:aux3} and \cref{eq:aux2} in \cref{eq:bounds_final1} gives us the following bound
	\begin{align}
		&\frac{1}{4\Delta t}\norm[1]{\sphih^n}_{L^4(\Omega)}^4 + \frac{1-\varepsilon}{2\Delta t}\norm[1]{\sphih^n}_{\mathcal{T}_h}^2 + \frac{1}{2\Delta t}\norm[1]{\spsih^n}^2_{\mathcal{T}_h} + \norm[1]{M^{1/2}(\sphih^{n-1})\bph^n}^2_{\mathcal{T}_h} \nonumber
		\\
		&+ \tau_3\norm[1]{\smuh^n - \fmuh^n}^2_{\partial\mathcal{T}_h}
		+ \frac{\tau_4}{\Delta t}\norm[1]{\sphih^n - \fphih^n}^2_{\partial\mathcal{T}_h} 
		\leq -\frac{2}{\Delta t}\elementprod{1}{\spsih^{n-1}}{\sphih^n} \nonumber
		\\
		&+ \frac{2}{\Delta t}\elementprod{1}{\spsih^n}{\sphih^{n-1}} 
		+ \Delta t^{-1}\faceprod{\tau_4(\sphih^n - \fphih^n)}{\sphih^{n-1} - \fphih^{n-1}} 
		\nonumber
		\\
		&+ \frac{1}{4\Delta t}\norm[1]{\sphih^{n-1}}_{L^4(\Omega)}^4 
		+ \frac{1-\varepsilon}{2\Delta t}\norm[1]{\sphih^{n-1}}_{\mathcal{T}_h}^2 + \frac{1}{2\Delta t}\norm[1]{\spsih^{n-1}}^2_{\mathcal{T}_h}.
	\end{align}
	Using Cauchy--Schwarz and Young's inequalities on the right hand side, we have
	\begin{align}
		&\frac{1}{4\Delta t}\norm[1]{\sphih^n}_{L^4(\Omega)}^4 + \frac{1-\varepsilon}{2\Delta t}\norm[1]{\sphih^n}_{\mathcal{T}_h}^2 + \frac{1}{2\Delta t}\norm[1]{\spsih^n}^2_{\mathcal{T}_h} + \norm[1]{M^{1/2}(\sphih^{n-1})\bph^n}^2_{\mathcal{T}_h} \nonumber
		\\
		&+ \tau_3\norm[1]{\smuh^n - \fmuh^n}^2_{\partial\mathcal{T}_h}
		+ \frac{\tau_4}{\Delta t}\norm[1]{\sphih^n - \fphih^n}^2_{\partial\mathcal{T}_h} 
		\leq 
		\frac{1}{\gamma_1\Delta t}\norm[1]{\spsih^{n-1}}_{\mathcal{T}_h}^2 + \frac{\gamma_1}{\Delta t}\norm[1]{\sphih^n}_{\mathcal{T}_h}^2
		\nonumber
		\\
		&+ \frac{\gamma_2}{\Delta t}\norm[1]{\spsih^n}^2_{\mathcal{T}_h} + \frac{1}{\gamma_2\Delta t}\norm[1]{\sphih^{n-1}}^2_{\mathcal{T}_h} 
		+ \frac{\tau_4}{2\gamma_3\Delta t}\norm[1]{\sphih^{n-1} - \fphih^{n-1}}^2_{\partial\mathcal{T}_h} \nonumber
		\\
		&+ \frac{\gamma_3\tau_4}{2\Delta t}\norm[1]{\sphih^n - \fphih^n}^2_{\partial\mathcal{T}_h}
		+ \frac{1}{4\Delta t}\norm[1]{\sphih^{n-1}}_{L^4(\Omega)}^4 + \frac{1-\varepsilon}{2\Delta t}\norm[1]{\sphih^{n-1}}_{\mathcal{T}_h}^2 + \frac{1}{2\Delta t}\norm[1]{\spsih^{n-1}}^2_{\mathcal{T}_h},
	\end{align}
	where we have again used that $M(\phi) \geq 0$. Taking $\gamma_1 = (1-\varepsilon)/4$, $\gamma_2 = 1/4$ and $\gamma_3 = 1$, gives the final result.
\end{proof}

Next, we show uniqueness of the solution.
\begin{theorem}[Uniqueness]
	If $\tau_1,\tau_2,\tau_3,\tau_4 >0$ and $\tau_2 = \tau_1$, the solution to \cref{eq:discrete_form} is unique.  
\end{theorem}
\begin{proof}
	For a given $\sphih^{n-1}, \spsih^{n-1} \in W_h,$
	let $(\br_i^n,\bq_i^n,\bp_i^n, \allowbreak \bs_i^n,\sphi_i^n,\spsi_i^n,\smu_i^n,\fphi_i^n,\fpsi_i^n,\fmu_i^n)$, $i=1,2 $ be two solutions to~\cref{eq:discrete_form} and let us denote their difference by $(\br^{1,2},\bq^{1,2},\bp^{1,2},\allowbreak \bs^{1,2}, \sphi^{1,2}, \spsi^{1,2}, \allowbreak\smu^{1,2}, \fphi^{1,2}, \fpsi^{1,2}, \fmu^{1,2})$, i.e., $\br^{1,2} = \br_1^n - \br_2^n$, and similar for the other variables. It is easy to see that the difference between two solutions satisfies the following system:
	\begin{subequations}
		\begin{align*}
			\elementprod{1}{\br^{1,2}}{\testr} - \elementprod{1}{\sphi^{1,2}}{\nabla \cdot \testr} + \faceprod{\fphi^{1,2}}{\testr \cdot \boldsymbol{n}} &= 0,
			\\
			\elementprod{1}{\nabla \cdot \br^{1,2}}{\testphi} +  \faceprod{\tau_1(\sphi^{1,2} - \fphi^{1,2})}{\testphi} + \elementprod{1}{\spsi^{1,2}}{\testphi} &= 0,
			\\
			\elementprod{1}{\bq^{1,2}}{\testq} - \elementprod{1}{\spsi^{1,2}}{\nabla\cdot \testq} + \faceprod{\fpsi^{1,2}}{\testq \cdot \boldsymbol{n}} &= 0,
			\\
			\elementprod{1}{\nabla \cdot \bq^{1,2}}{\testpsi} + \faceprod{\tau_2(\spsi^{1,2} - \fpsi^{1,2})}{\testpsi} &
			\nonumber
			\\
			+ \elementprod{1}{\smu^{1,2}}{\testpsi} - (1-\varepsilon)\elementprod{1}{\sphi^{1,2}}{\testpsi} -\faceprod{\tau_4(\sphi^{1,2} - \fphi^{1,2})}{\testpsi} &= \elementprod{1}{(\sphi_1^n)^3 - (\sphi_2^n)^3}{\testpsi},
			\\
			\elementprod{1}{\bp^{1,2}}{\testp} - \elementprod{1}{\smu^{1,2}}{\nabla \cdot \testp} + \faceprod{\fmu^{1,2}}{\testp \cdot \boldsymbol{n}} &= 0,
			\\
			\elementprod{1}{\bs^{1,2}}{\tests} - \elementprod{1}{M(\sphih^{n-1})\bp^{1,2}}{\tests} &= 0,
			\\
			\elementprod{1}{\nabla \cdot \bs^{1,2}}{\testmu} + \faceprod{\tau_3(\smu^{1,2} - \fmu^{1,2})}{\testmu} + \Delta t^{-1}\elementprod{1}{\sphi^{1,2}}{\testmu} &= 0,
			\\
			\faceprod{\br^{1,2} \cdot \boldsymbol{n} + \tau_1(\sphi^{1,2} - \fphi^{1,2})}{\testfphi} &= 0,
			\\
			\faceprod{\bq^{1,2} \cdot \boldsymbol{n} + \tau_2(\spsi^{1,2} - \fpsi^{1,2}) - \tau_4(\sphi^{1,2} - \fphi^{1,2})}{\testfpsi} &= 0,
			\\
			\faceprod{\bs^{1,2} \cdot \boldsymbol{n} + \tau_3(\smu^{1,2} - \fmu^{1,2})}{\testfmu} &= 0.
		\end{align*}
		\label{eq:discrete_form_uniq}
	\end{subequations}
	for all $(\testr,\testq,\testp,\tests,\testphi,\testpsi,\testmu,\testfphi,\testfpsi,\testfmu) \in \boldsymbol{V}_h^4\times W_h^3\times M_h^3$.
	The relation \cref{eq:relation_linear} gives
	\begin{align}
		\norm[1]{\spsi^{1,2}}^2_{\mathcal{T}_h} + (1-\varepsilon)\norm[1]{\sphi^{1,2}}^2_{\mathcal{T}_h} + \Delta t\norm[1]{M^{1/2}(\sphih^{n-1}) \bp^{1,2}}^2_{\mathcal{T}_h} + \tau_3 \Delta t \norm[1]{\smu^{1,2} - \fmu^{1,2}}^2_{\partial\mathcal{T}_h} 
		\nonumber
		\\
		+ \tau_4 \norm[1]{\sphi^{1,2} - \fphi^{1,2}}^2_{\partial\mathcal{T}_h} = -\elementprod{1}{(\sphi_1^n)^3 - (\sphi_2^n)^3}{\sphi^{1,2}}.
		\label{eq:aux_uniq1}
	\end{align}
	Note that, since $\phi^{1,2} = \sphi_1^n - \sphi_2^n$, we have
	\begin{align}
		-\elementprod{1}{(\phi_1^n)^3-(\phi_2^n)^3}{\sphi^{1,2}}
		&= -\frac{1}{2}\elementprod{1}{(\phi_1^n+\phi_2^n)^2\sphi^{1,2}
			+ ((\phi_1^n)^2+(\phi_2^n)^2)\sphi^{1,2}  
		}{\sphi^{1,2}}
		\notag
		\\
		&= -\frac{1}{2}\elementprod{1}{(\phi_1^n+\phi_2^n)^2
			+
			((\phi_1^n)^2+(\phi_2^n)^2)  
		}{(\sphi^{1,2})^2}
		\notag
		\\ 
		&\leq 0.
		\label{eq:uni-proof02}
	\end{align}
	Thus, \cref{eq:aux_uniq1} becomes
	\begin{align}
		\norm[1]{\spsi^{1,2}}^2_{\mathcal{T}_h} + (1-\varepsilon)\norm[1]{\sphi^{1,2}}^2_{\mathcal{T}_h} + \Delta t\norm[1]{M^{1/2}(\sphih^{n-1}) \bp^{1,2}}^2_{\mathcal{T}_h} 
		\nonumber
		\\
		+ \tau_3 \Delta t \norm[1]{\smu^{1,2} - \fmu^{1,2}}^2_{\partial\mathcal{T}_h} 
		+ \tau_4 \norm[1]{\sphi^{1,2} - \fphi^{1,2}}^2_{\partial\mathcal{T}_h} \leq 0.
		\label{eq:aux_uniq2}
	\end{align}
	Since $\varepsilon < 1$, $\tau_4, \tau_3 > 0$, and $\Delta t > 0$, we then have
	\begin{equation*}
		\sphi^{1,2} = \spsi^{1,2} = 0 \text{ on }\mathcal{T}_h, \quad \fphi^{1,2} = 0 \text{ on }\mathcal{E}_h, \quad \fmu^{1,2} = \smu^{1,2} \text{ on }\mathcal{E}_h, \quad \bp^{1,2} = \boldsymbol{0} \text{ on }\mathcal{T}_h.
	\end{equation*}
	As in the proof of well-posedness in \cref{lem:wellposed_cont}, we can conclude that all the other variables are zero, and therefore the solution is unique.
\end{proof}
\section{Numerical results}\label{sec:numerical_results}

In this section, we present several test cases to illustrate the properties of the method that we have shown in the previous sections. Since the PFC model operates at the atomic length and diffusive length scales, it is known to effectively simulate the formation and evolution of crystalline structures, as well as the ensuing defects and grain boundaries in crystalline materials. 
We demonstrate the proposed method's ability in capturing these behaviors in the test examples presented in ~\cref{ss:crystal_formation,ss:benchmark,ss:phase_transformation,ss:grain_grows}.
We have implemented the discretization in the Modular Finite Element Method (MFEM) library \cite{mfem}. The linear systems are solved using MUMPS \cite{MUMPS2001,MUMPS2006} through PETSc \cite{balay_petsctao_2023,balay_efficient_1997}.

We used Newton linearization to handle the nonlinearity of \eqref{eq:discrete_form}. We used the solution from the previous time step (and the initial condition at the first step) as an initial guess. The stopping criterion for the iterations was that the Euclidean norm of the residual vector obtained from the \eqref{eq:discrete_form} had to be below $10^{-8}$ in all simulations.

We achieve the solution of the linearized system in two steps: first, we solve the static condensed system to obtain a Newton update on the skeletal unknowns; second, we perform an element-wise reconstruction to obtain the Newton update for the volume unknowns. This is the standard approach for HDG and EDG methods; see e.g. \cite{horvath:2020,Nguyen:2009}.

Unless stated otherwise, we use the stability parameters as $\tau_1 = \tau_2 = \tau_3 = \frac12\tau_4 = 10$. We used the non-degenerate case, $M = 1$, for most examples.

\subsection{Accuracy of the scheme}\label{ss:accuracy}

First, we test the rates of convergence of our method with a manufactured solution. We consider the domain $\Omega = (0,2\pi)^2$, and periodic boundary conditions. Moreover, we add a source term to \cref{eq:new_variables}, and set the initial conditions such that the exact solution of the problem is
\begin{equation*}
	\phi(x,y,t) = e^{-2t}\sin(x)\sin(y).
\end{equation*}

We take $\varepsilon = 0.5$ and the final time $T=1$ for these simulations. The $L^2(\Omega)$ errors at the final time level are shown for different $\Delta t$ and $h$ values in \cref{tab:periodic} for the non-degenerate case, $M = 1$, and in \cref{tab:periodic_degen} for the degenerate case, $M = 1-\sphi^2$. We observe a rate of convergence of 1, as expected from the first-order time discretization. \Cref{tab:periodic} and \cref{tab:periodic_degen} show the results for EDG discretization. However, the errors for the HDG and the EDG methods are identical up to two meaningful decimal places.

\Cref{fig:error_vs_dofs} shows the error values from \cref{tab:periodic} for the case of $\Delta t / h = 0.95$ as a function of the number of degrees of freedom, both for the EDG and the HDG discretizations. For the convergence tests, we use Cartesian quadrilateral grids, and we observe that the number of degrees of freedom for the HDG discretization matches that of the EDG discretization after one additional refinement. So, the HDG discretization requires approximately four times as many unknowns to achieve the same error as the EDG one. This shows the superiority of the EDG discretization for these problems.

\begin{table}[!ht]
	\centering
	\begin{tabular}{c|c|c||c|c||c|c}
		N   &  $\|\sphi - \sphih\|_{\mathcal{T}_h}$   & Rate &  $\|\sphi - \sphih\|_{\mathcal{T}_h}$   & Rate &  $\|\sphi - \sphih\|_{\mathcal{T}_h}$   & Rate \\
		48  & 1.20E+00 &   -   & 5.66E-01 &    -  & 2.74E-01 &    -  \\
		96  & 8.74E-01 & 0.45 & 3.33E-01 & 0.76 & 1.48E-01 & 0.88 \\
		192 & 5.68E-01 & 0.62 & 1.82E-01 & 0.87 & 7.71E-02 & 0.94 \\
		384 & 3.33E-01 & 0.77 & 9.56E-02 & 0.93 & 3.93E-02 & 0.97 \\
		768 & 1.82E-01 & 0.87 & 4.89E-02 & 0.97 & 1.98E-02 & 0.99 \\
	\end{tabular}
	\caption{Non-degenerate case ($M = 1$). $L^2(\Omega)$ errors at the final time level. Left: $\Delta t / h = 3.82$, middle: $\Delta t / h = 0.95$, right: $\Delta t / h= 0.382$.}
	\label{tab:periodic}
\end{table}

\begin{table}[!ht]
	\centering
	\begin{tabular}{c|c|c||c|c||c|c}
		N   &  $\|\sphi - \sphih\|_{\mathcal{T}_h}$   & Rate &  $\|\sphi - \sphih\|_{\mathcal{T}_h}$   & Rate &  $\|\sphi - \sphih\|_{\mathcal{T}_h}$   & Rate \\
		48  & 1.22E+00 &   -   & 5.88E-01 &     - & 2.87E-01 &   -  \\
		96  & 8.99E-01 & 0.44 & 3.49E-01 & 0.75 & 1.57E-01 & 0.87\\
		192 & 5.90E-01 & 0.61 & 1.93E-01 & 0.85 & 8.20E-02 & 0.94\\
		384 & 3.50E-01 & 0.75 & 1.01E-01 & 0.93 & 4.18E-02 & 0.97\\
		768 & 1.93E-01 & 0.86 & 5.20E-02 & 0.96 & 2.12E-02 & 0.98
	\end{tabular}
	\caption{Degenerate case ($M = 1-\sphi^2$). $L^2(\Omega)$ errors at the final time level. Left: $\Delta t / h = 3.82$, middle: $\Delta t / h = 0.95$, right: $\Delta t / h = 0.382$.}
	\label{tab:periodic_degen}
\end{table}

\begin{figure}[!h!t]
	\centering
	\includegraphics[width=0.8\textwidth]{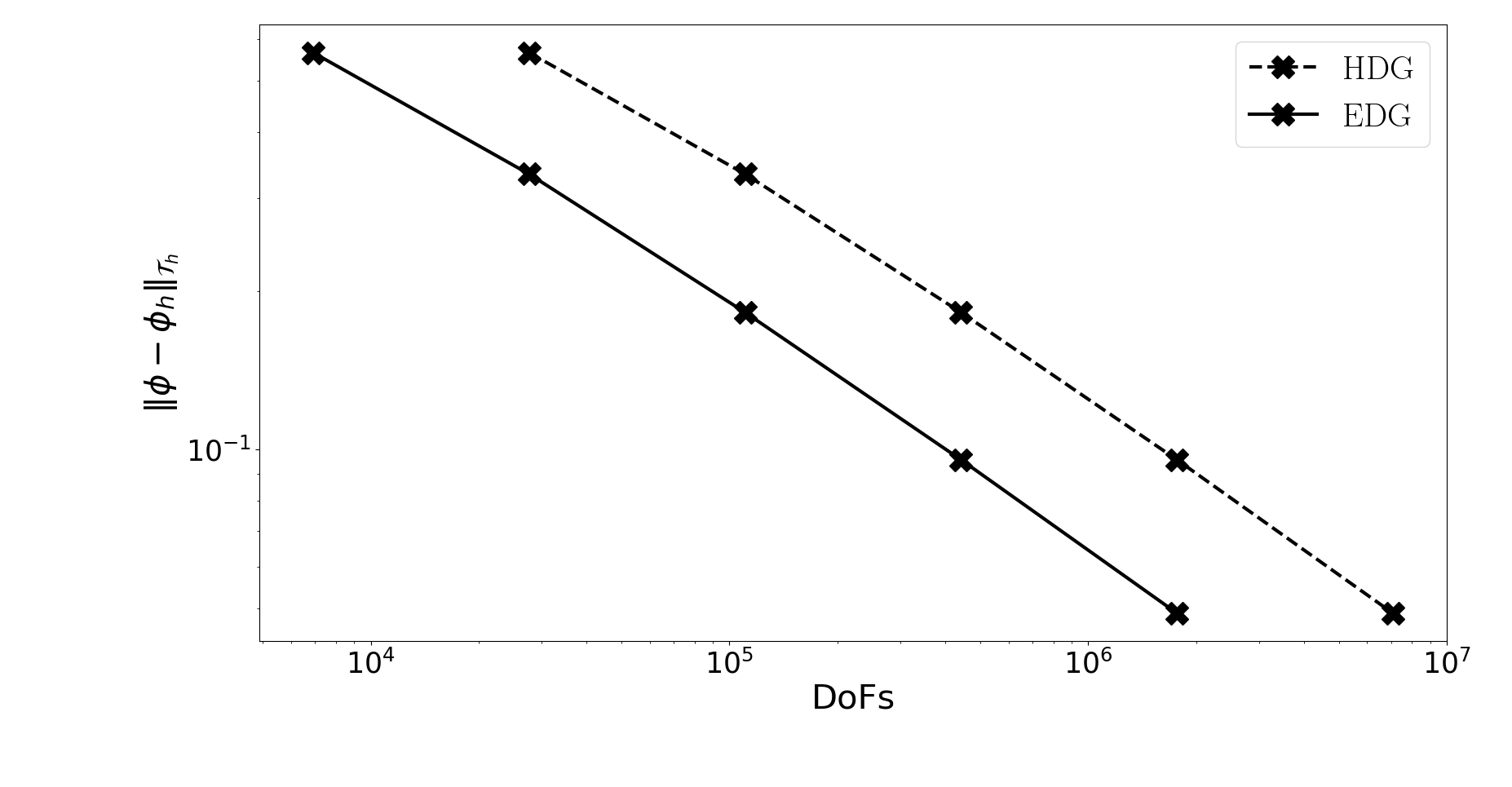}
	\caption{The error as a function of the number of degrees of freedom. The error values are taken from \cref{tab:periodic}, for $\Delta t / h = 0.95$. Continuous line EDG, dashed line HDG discretization.}
	\label{fig:error_vs_dofs}
\end{figure}

\subsection{Growth of a Monocrystal in a supercooled liquid}\label{ss:crystal_formation}
In this example, we demonstrate the ability of the method in capturing the liquid-to-solid phase transition through the growth of a single crystal (solid) in a supercooled liquid. Here, $\phi(x,y,t)$ represents a local atomistic density field that describes this two-phase system. It is assumed to have a uniform value $\phi_a$ in the liquid phase while exhibiting the crystal lattice's symmetry and periodicity in the spatially dependent solid phase.  To initiate nucleation, we place a single seed crystal with spatially varying density $\phi_S(x,y)$ at the center $(x_0,y_0)$ of our computational domain $\Omega$. The computational domain is assumed to be occupied by a supercooled liquid with homogeneous density denoted by a constant $\phi_a$ as described in~\eqref{def:monocrystal}. We present the numerical performance of this example primarily to benchmark against the previously published simulations~\cite{calo2020splitting,vignal2017energy} and the references therein. To the best of the authors' knowledge, specific details of the type of underlying crystalline material are unavailable in the literature.
The computational domain is $\Omega = \left(0, \frac{36\pi}{\sqrt{3}}\right) \times \left(0, 24\pi\right)$ with periodic boundary conditions. The initial condition is given by
\begin{equation}\label{def:monocrystal}
	\sphi(x,y,0) = \sphi_{a} + \omega(x,y)A\sphi_S(x,y),
\end{equation}
where 
\begin{align*}
	\sphi_S(x,y) &= \cos\left(\frac{qy}{\sqrt{3}}\right)\cos\left(qx\right) - 0.5 \cos\left(\frac{2qy}{\sqrt{3}}\right),\\
	\omega(x,y) &=\begin{cases}
		\left(1-\left(\displaystyle \frac{\|(x,y)-(x_0,y_0)\|}{d_0}\right)^2\right)^2 \quad & \mbox{if\,\,} \|(x,y)-(x_0,y_0)\|\le d_0,\\
		0 \quad &\mbox{otherwise.}
	\end{cases}
\end{align*}
The parameters are
\begin{equation*}
	A = \frac{4}{5}\left(\sphi_{a} + \frac{\sqrt{15\varepsilon - 36\bar{\sphi}^2}}{3}\right),
	\quad \varepsilon = 0.325, \quad  \sphi_{a} = \frac{\sqrt{\varepsilon}}{2}, \quad q=\frac{\sqrt{3}}{2},
\end{equation*}
while $(x_0,y_0)$ is the center of the domain, and $d_0$ is 1/6 of the domain length in the $x$-direction. The crystal growth is shown in \cref{fig:Chrystal_growth} for different times, using a $460 \times 532$ mesh, with time step size $\Delta t = 0.01$. We can see the development of the crystal growth at different times. The results match those found in the literature; see \cite{calo2020splitting,vignal2017energy} and the references therein.
In~\cref{fig:Chrystal_growth}, we observe the evolution of the conserved dynamics of transitioning from a liquid to a solid phase in the presence of a seed crystal placed in the center of the domain. We note that our method effectively captures this phase transition and the ensuing periodic lattice structure inherited by the crystal phase.
\begin{figure}
	\centering
	\includegraphics[width=0.3\textwidth]{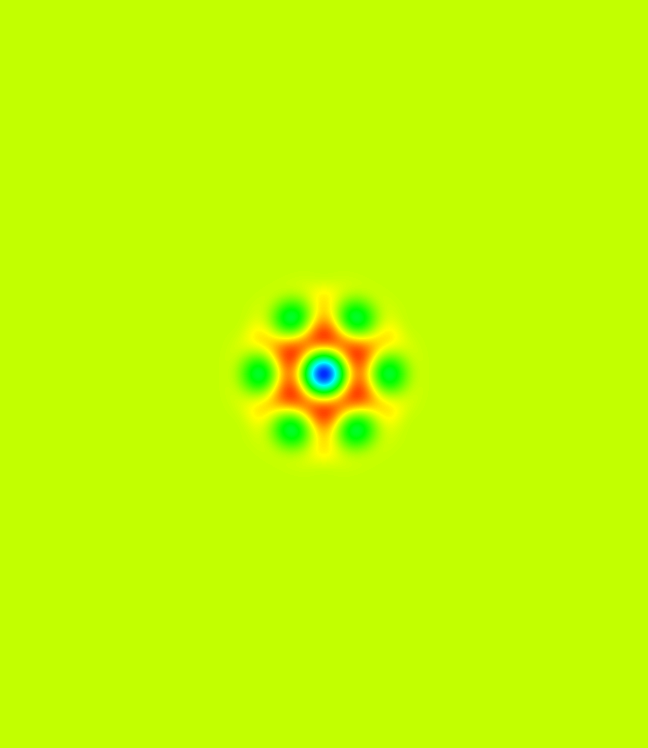}
	\includegraphics[width=0.3\textwidth]{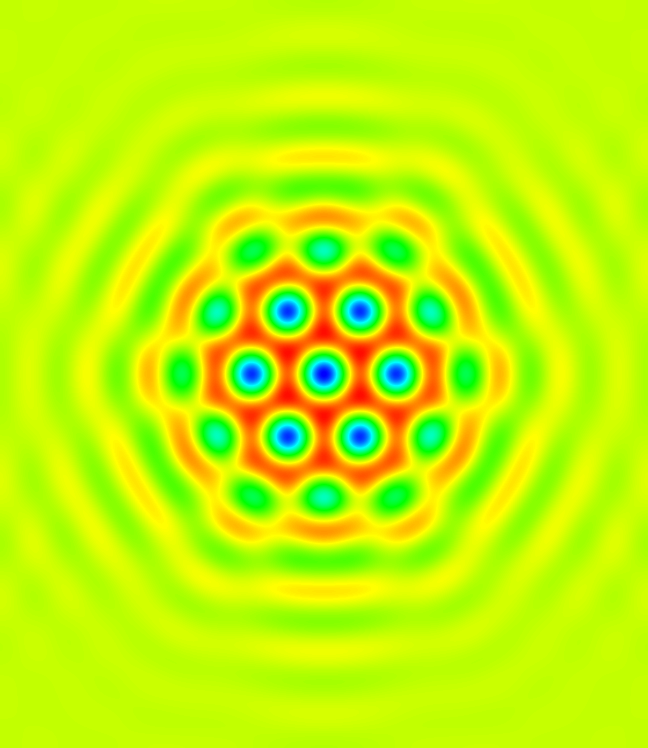}
	\includegraphics[width=0.3\textwidth]{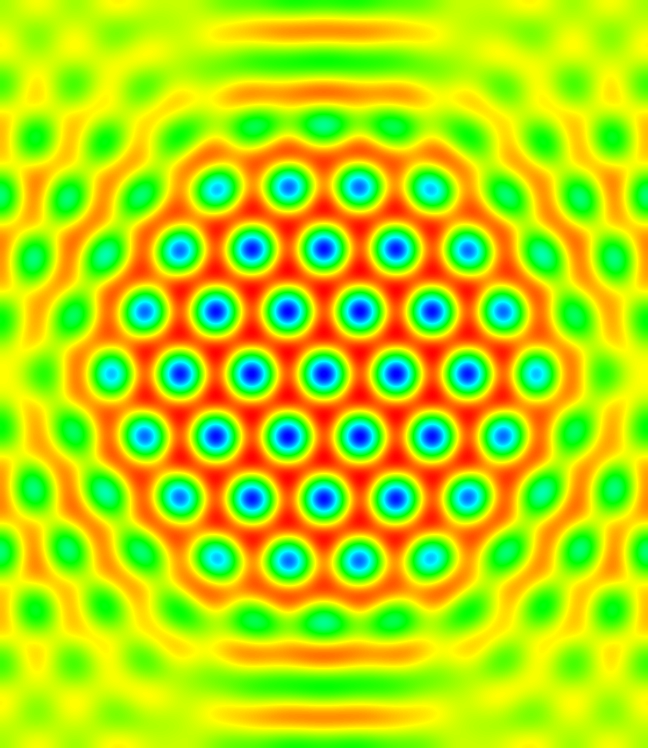}\\
	\includegraphics[width=0.3\textwidth]{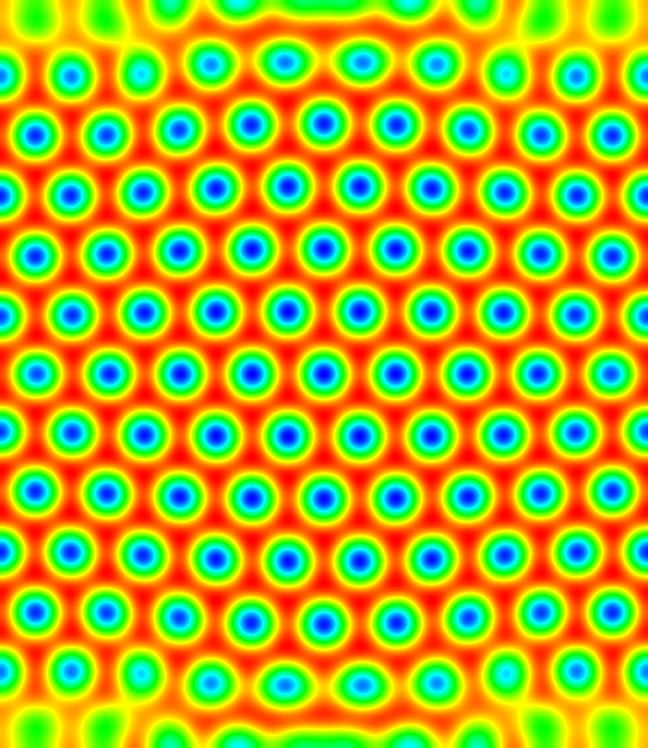}
	\includegraphics[width=0.3\textwidth]{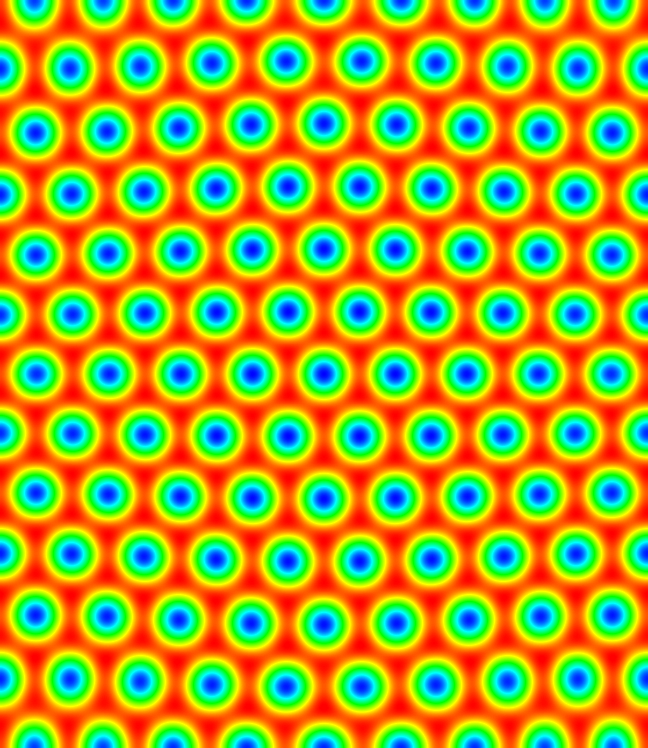}
	\caption{Crystal growth on a rectangular domain $\left(0, \frac{36\pi}{\sqrt{3}}\right) \times \left(0, 24\pi\right)$ using a mesh consisting of  $460 \times 532$ elements, and a time step size $\Delta t = 0.01$. The times shown are $t=0, 20, 40, 60, 80$ (from left to right, top to bottom). The colors are scaled from $\sphi = -0.67$ to $\sphi = 0.72$.  See \cref{ss:crystal_formation} for more details.}
	\label{fig:Chrystal_growth}
\end{figure}

\subsection{Benchmark}\label{ss:benchmark}

We implement a problem previously introduced in \cite{HWWL:09:PFC} and also analyzed in \cite{DS:23:C0PFC}. The computational domain is $\Omega = \left(0, 32\right)^2$, with homogeneous Neumann boundary conditions. The initial condition is given by
\begin{align*}
	\sphi&(x, y) = 0.07 - 0.02 \cos\left(\frac{2\pi(x - 12)}{32}\right)
	\sin\left(\frac{2\pi(y - 1)}{32}\right)\\
	&+ 0.02 \cos^2\left(\frac{\pi(x + 10)}{32}\right)
	\cos^2\left(\frac{\pi(y + 3)}{32}\right)
	- 0.01 \sin^2\left(\frac{4\pi x}{32}\right)
	\sin^2\left(\frac{4\pi(y - 6)}{32}\right).
\end{align*}

We set $\varepsilon = 0.025$, and use a mesh containing $256 \times 256$ elements, with time step size $\Delta t = 0.005$. \Cref{fig:benchmark} shows the solution at different times, while \cref{fig:benchmark_total_energy} shows the decay of the total scaled energy, so $E/32^2$ until the final time $t = 10$. The results are in agreement with those found in \cite{DS:23:C0PFC, HWWL:09:PFC}.

\begin{figure}
	\centering
	\includegraphics[width=0.7\textwidth]{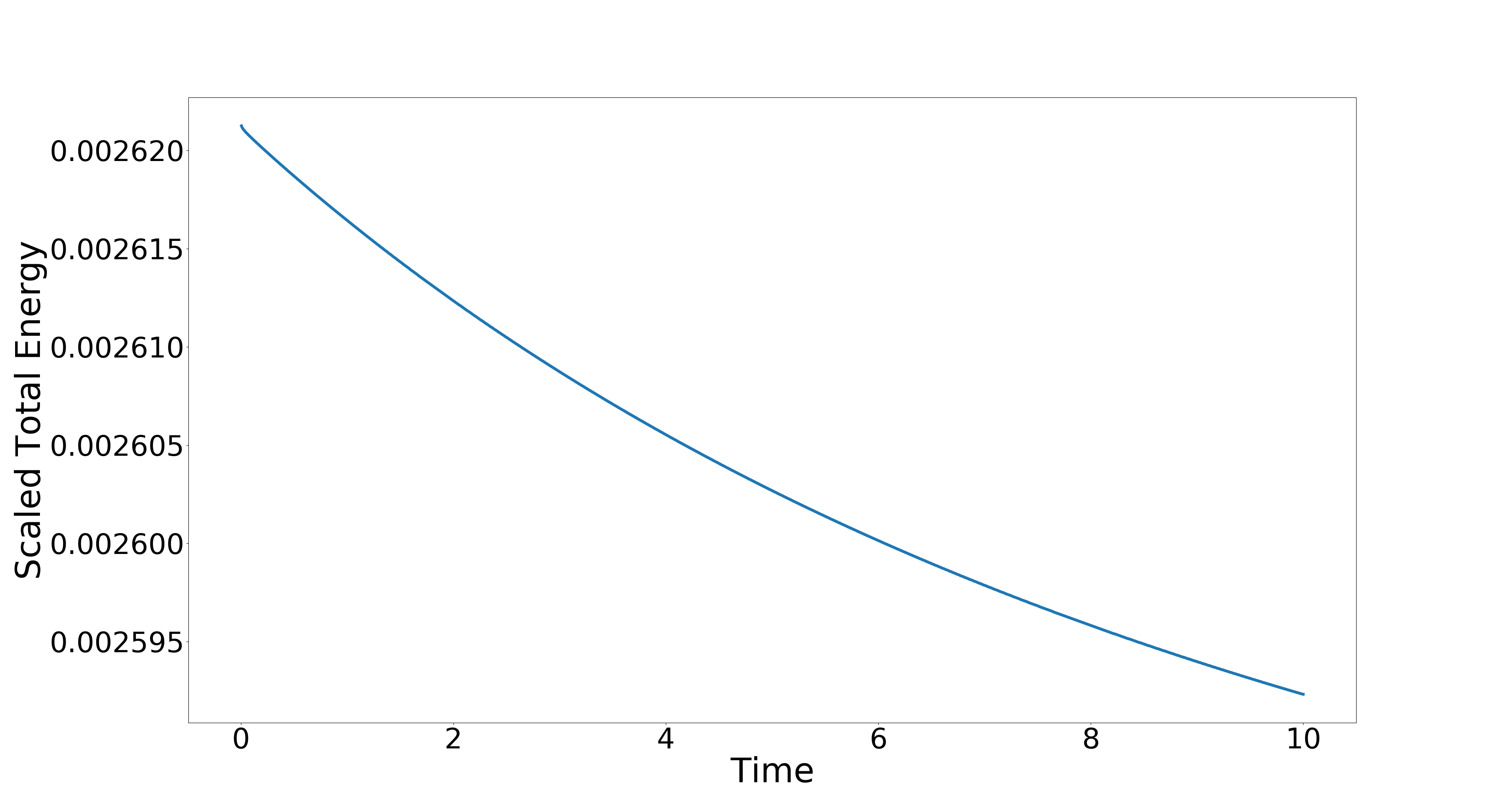}
	\caption{Benchmark test case, the decrease of the total scaled energy.}
	\label{fig:benchmark_total_energy}
\end{figure}

\begin{figure}
	\centering
	\includegraphics[width=0.3\textwidth]{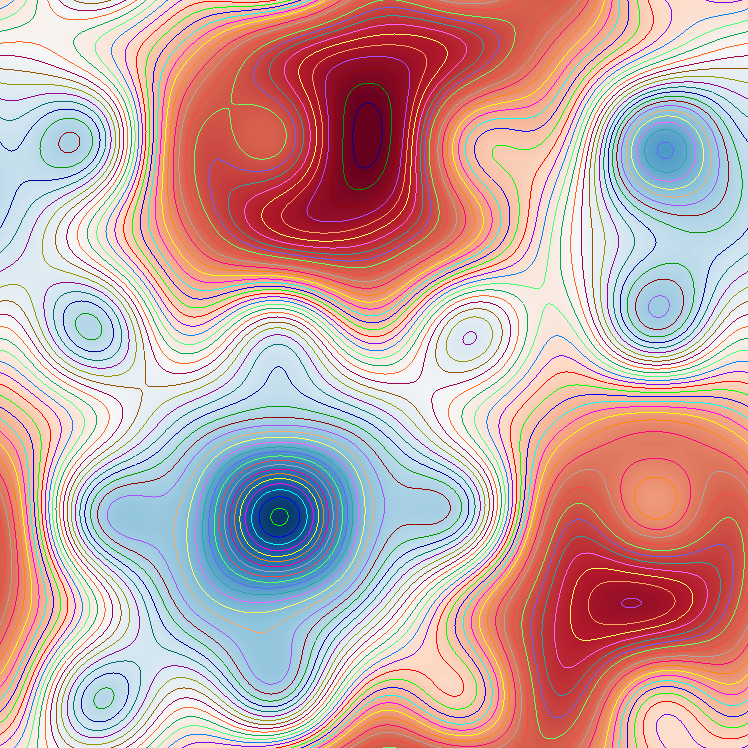}
	\includegraphics[width=0.3\textwidth]{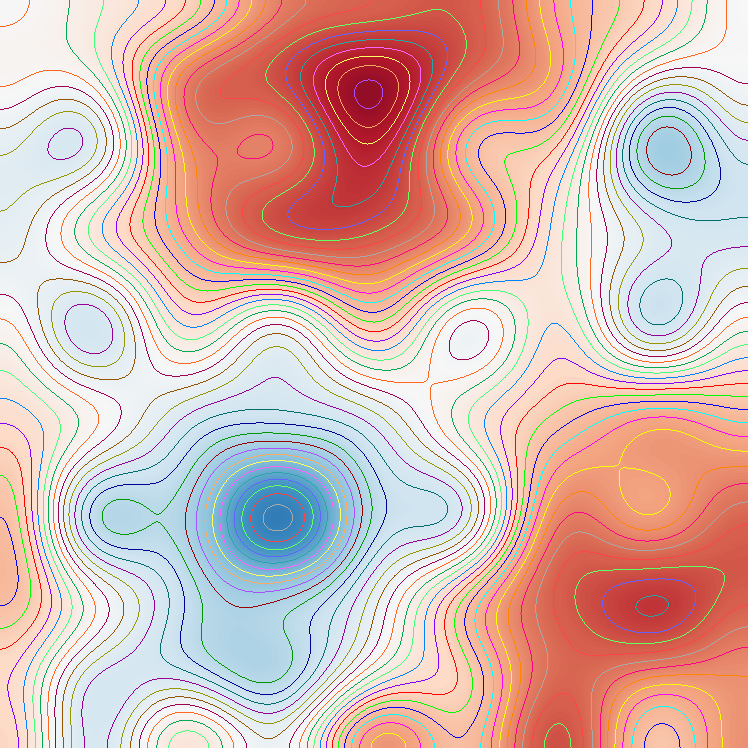}
	\includegraphics[width=0.3\textwidth]{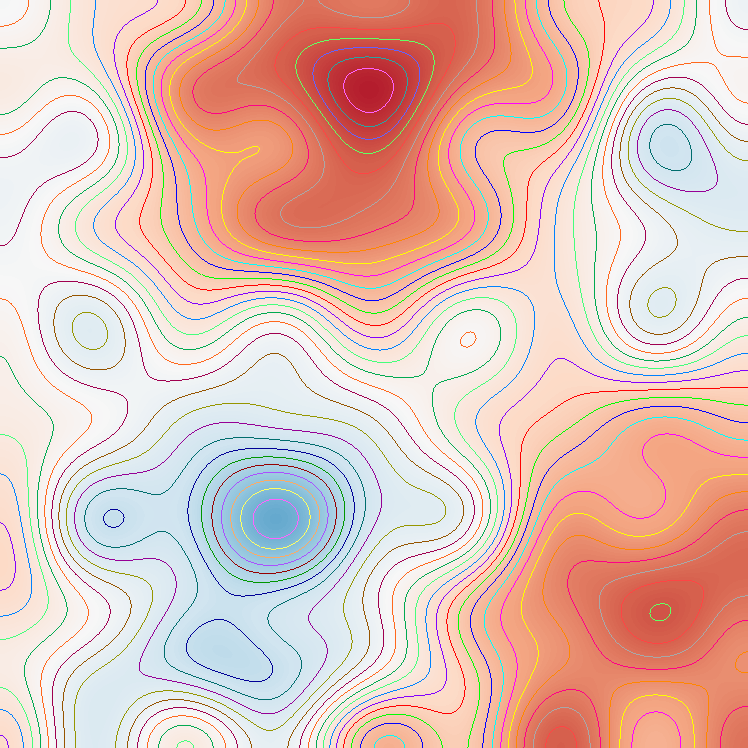}
	\caption{Benchmark test case, output at different times. The times shown are $t=0, 5, 10$ (from left to right).  The domain is $(0,32)^2$, the mesh consists of $256 \times 256$ elements, and $\Delta t = 0.005$. The colors are scaled from a density value of $\sphi = 0.04$ to $\sphi = 0.096$. See \cref{ss:benchmark} for more details.}
	\label{fig:benchmark}
\end{figure}

\subsection{Polycrystal Growth and Grain Boundary Dynamics}\label{ss:grain_grows}
In continuation of capturing the liquid-to-solid phase transitions presented in~\cref{ss:crystal_formation}, we again consider a supercooled liquid of uniform density $\bar{\sphi}$ occupying the computational domain $\Omega$. Instead of capturing the growth of a single seed crystal, we capture the growth of three seed crystals, placed at random locations $(x_l,y_l)$ within $\Omega$ and each with a different orientation. The different orientations of the seed crystals lead to the formation of defects such as vacancies and grain boundaries. The goal of this example is to demonstrate the method's ability to capture polycrystalline growth and the ensuing defects, which largely govern the macroscopic properties (e.g., tensile strength and material hardness) of the underlying crystalline materials. 
In this example, we initiate the nucleation of three crystals, each with a different orientation and
each placed at random locations within $\Omega$. As in~\cref{ss:crystal_formation}, there is a paucity of information regarding the underlying crystalline material considered in this benchmark example. The computational domain is  $\Omega = \left(0, 201\right)^2$ with Neumann boundary conditions, and the initial condition for the density is similar to the one given in \cref{ss:crystal_formation}  with the random locations of the three seed crystals described by:
\begin{equation}\label{def:polycrystal}
	\sphi(x_l,y_l) = \bar{\sphi} + C\cos\left(\frac{qy_l}{\sqrt{3}}\right)\cos\left(qx_l\right) - 0.5 \cos\left(\frac{2qy_l}{\sqrt{3}}\right)
\end{equation}

\noindent where $\bar{\sphi} = 0.285, C = 0.446, q = 0.66$. We use the method parameter $\varepsilon = 0.25$. The local coordinates $(x_l,y_l)$ are obtained from the global coordinates $(x,y)$ via a rotation by angle $\alpha$. We place three grains to the squares $[25, 40]^2$, $[60, 75]^2$, and $[85, 100]\times [25, 40]$,  with $\alpha = \pi/4, 0, -\pi/4$, respectively.
As nucleation of these three crystals occurs, we observe the evolution of the crystalline phase as well as the appearance of the defects in the form of dislocations and grain boundaries as described in~\cref{fig:grain_growth}.
\Cref{fig:grain_growth} shows the solution at different times using a mesh consisting of $402 \times 402$ elements and $\Delta t = 1$. We can see the crystals grow and develop well-defined crystal-liquid interfaces. This evolution is in excellent agreement with the results reported in
\cite{GN:2012:PFC, yang-han2017, DS:23:C0PFC, guo2016ldg}.

\begin{figure}
	\centering
	\includegraphics[width=0.3\textwidth]{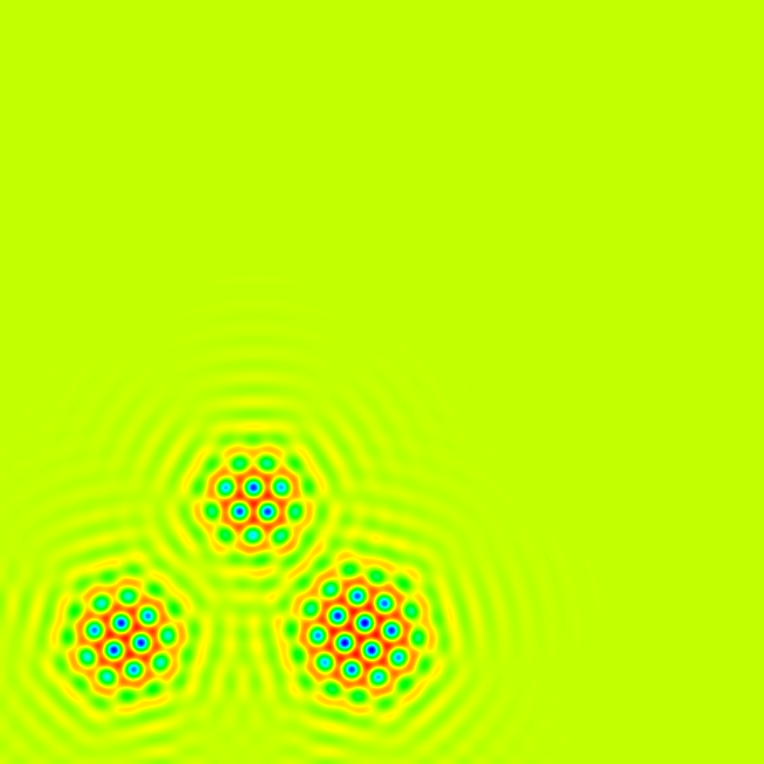}
	\includegraphics[width=0.3\textwidth]{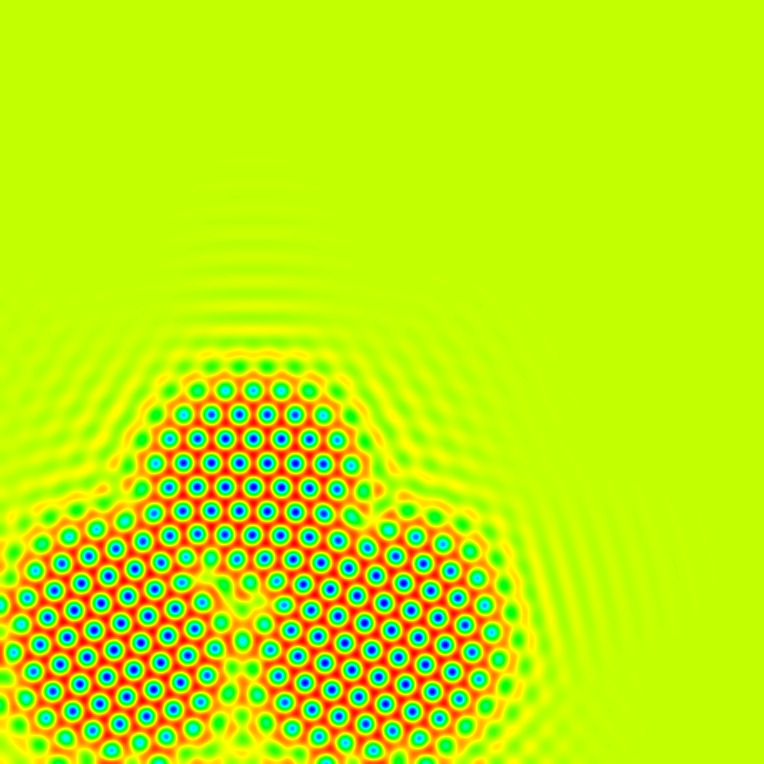}
	\includegraphics[width=0.3\textwidth]{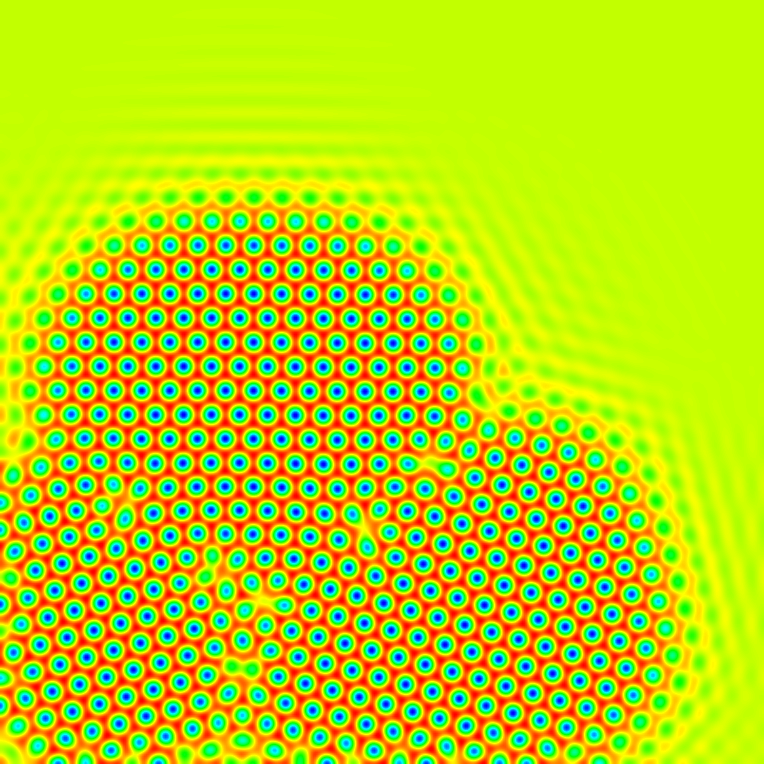}\\
	\includegraphics[width=0.3\textwidth]{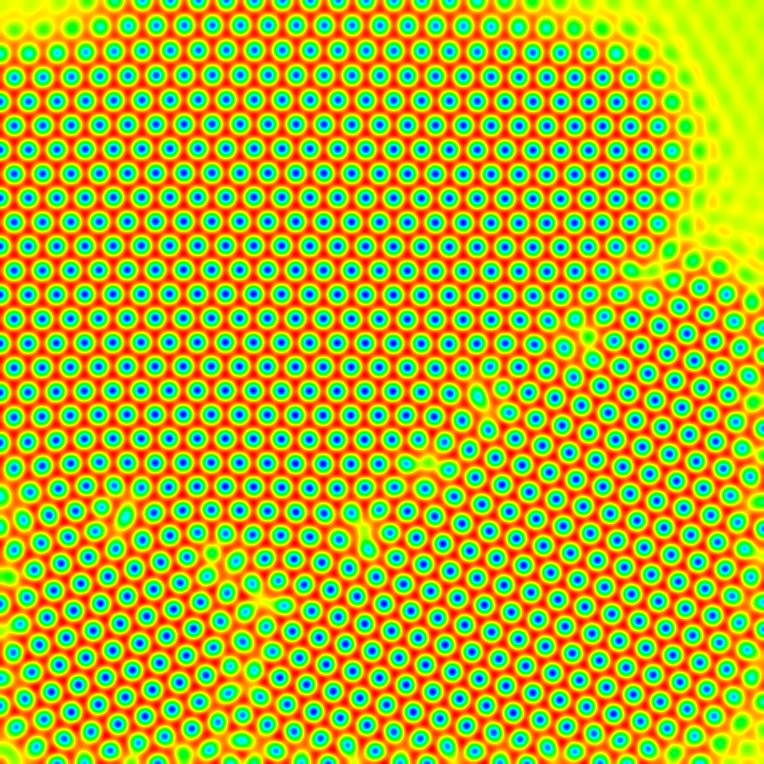}
	\includegraphics[width=0.3\textwidth]{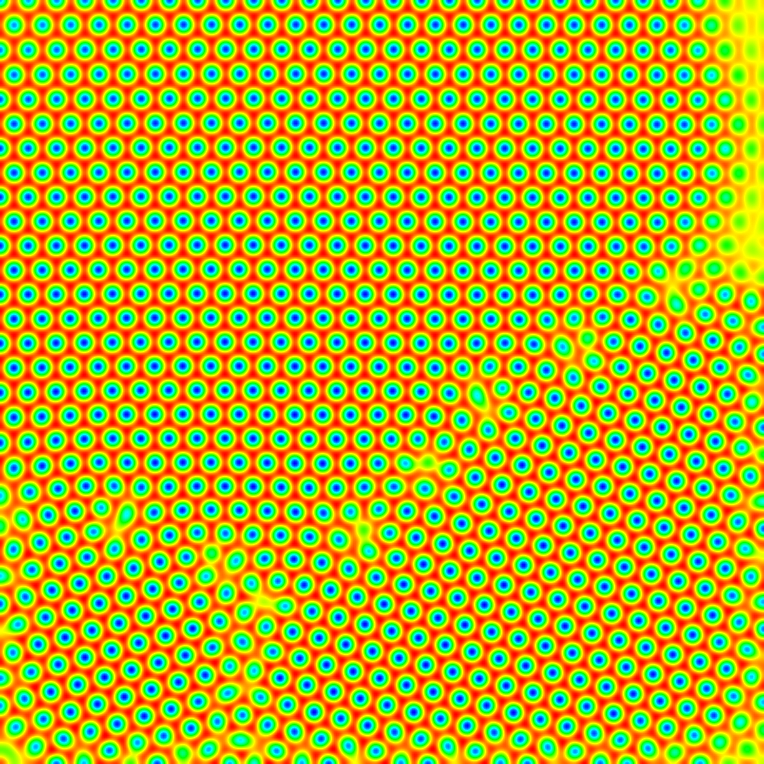}
	\includegraphics[width=0.3\textwidth]{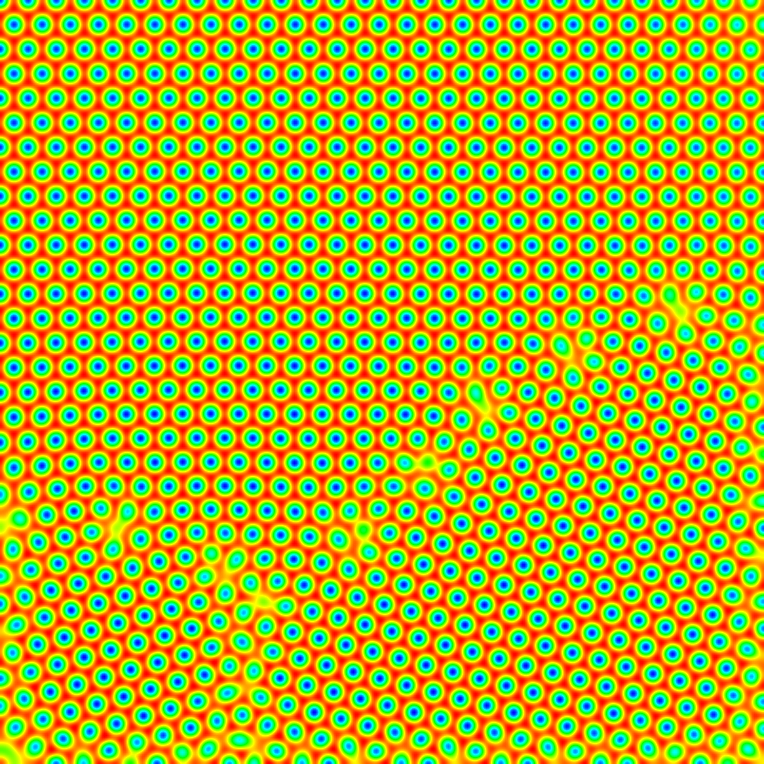}
	\caption{Atomic density for grain growth at times $t=250, 500, 1000, 2000, 3000, 4000$ (from left to right, top to bottom).  The domain is $(0,201)^2$, the mesh consists of $402 \times 402$ elements, and $\Delta t = 1$. The colors are scaled from $\sphi = -0.47$ to $\sphi = 0.63$. See \cref{ss:grain_grows} for more details.}
	\label{fig:grain_growth}
\end{figure}

\subsection{Phase transition}\label{ss:phase_transformation}

In the next experiment, the computational domain is $(0, 128)^2$ and we consider periodic boundary conditions. The initial density is $0.7 + \eta(x,y)$, where $\eta(x,y)$ is a uniformly distributed random number satisfying $-0.7 \le \eta(x,y) \le 0.7$. In this case, $\spsi_h^0$ is set to 0. The model parameter is $\varepsilon = 0.025$. The time step is $\Delta t = 0.01$, and the mesh consists of $256 \times 256$ elements, and we used biquadratic basis functions, similar to \cite{guo2016ldg}. The results at different times are shown in \cref{fig:phase_transition}. We can see similar phase transitions as in \cite{guo2016ldg,yang-han2017}.

\begin{figure}
	\centering
	\includegraphics[width=0.3\textwidth]{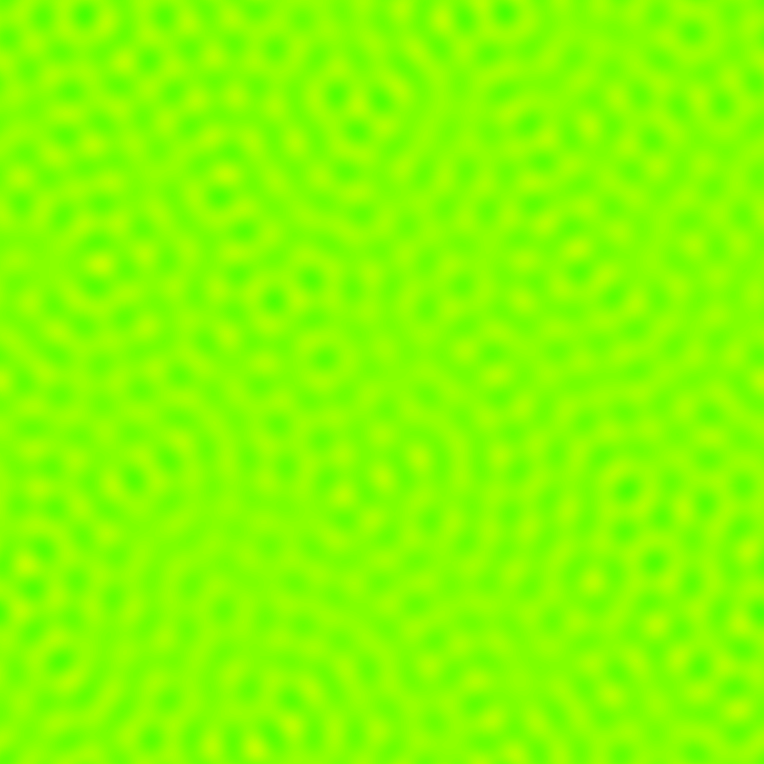}
	\includegraphics[width=0.3\textwidth]{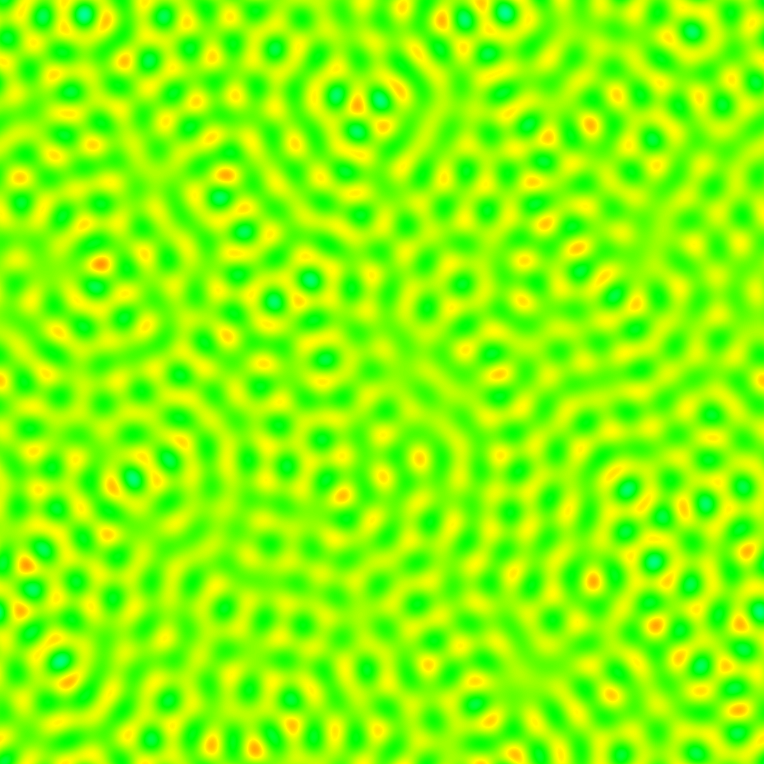}
	\includegraphics[width=0.3\textwidth]{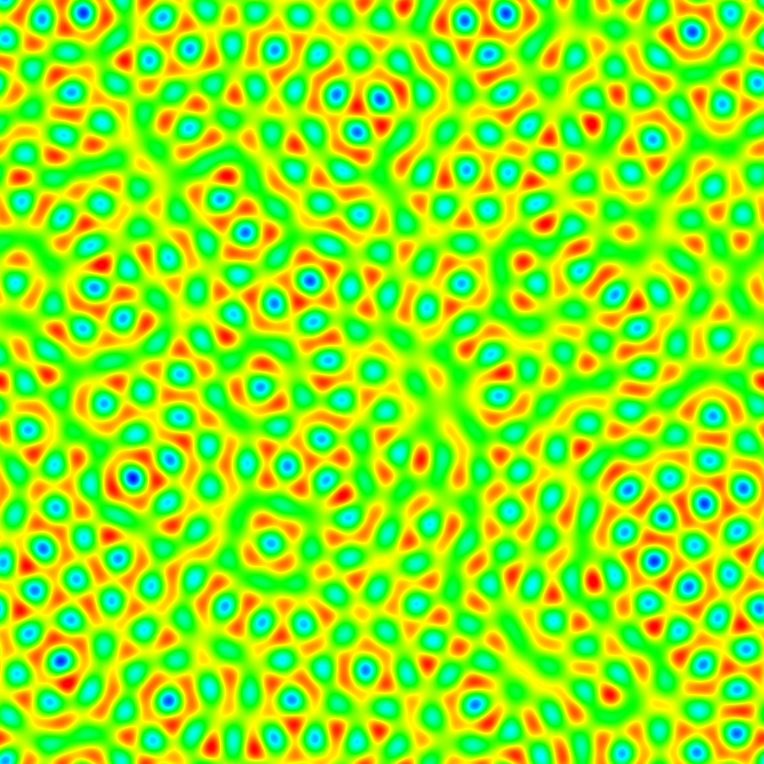}\\
	\includegraphics[width=0.3\textwidth]{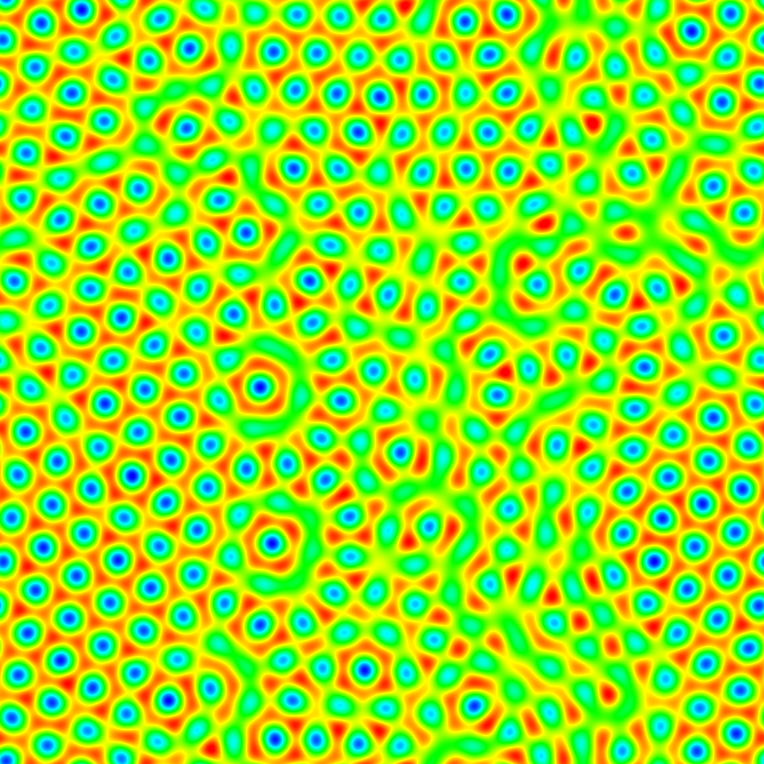}
	\includegraphics[width=0.3\textwidth]{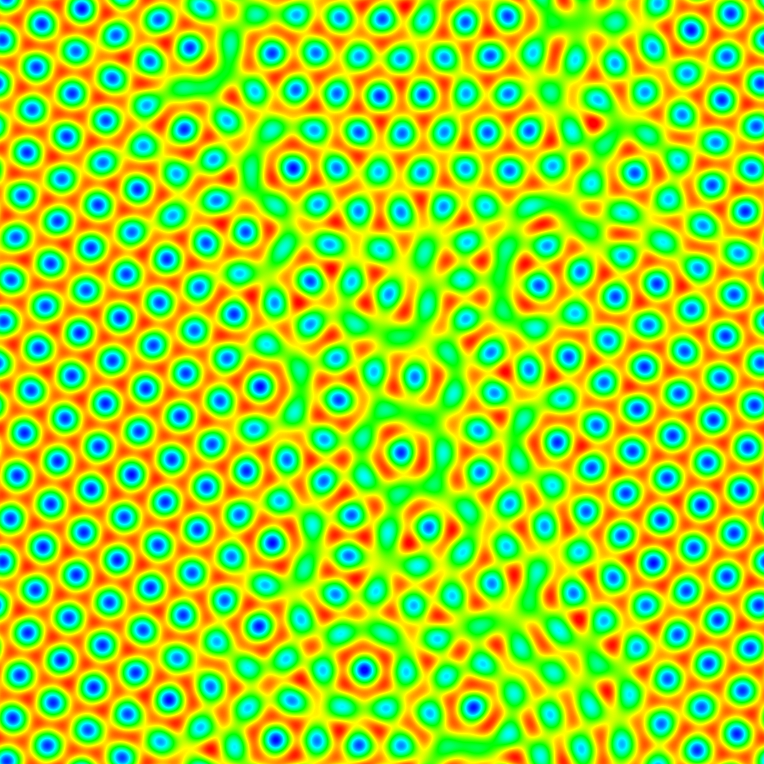}
	\includegraphics[width=0.3\textwidth]{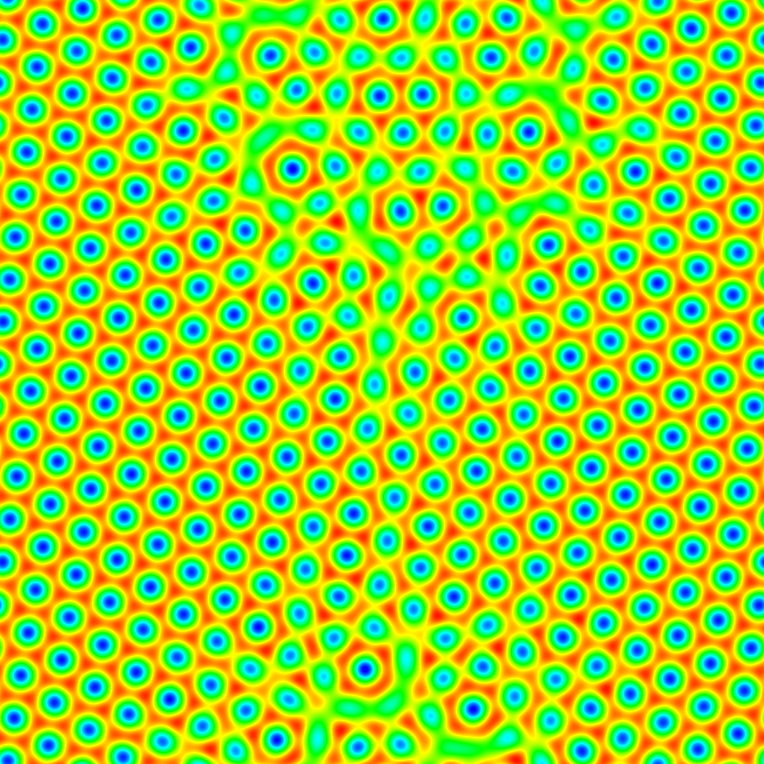}
	\caption{Crystal growth in a super-cooled liquid. The computational domain is $\left(0,128\right)^2$ using a mesh consisting of  $256 \times 256$ elements, with biquadratic basis functions and a time step size $\Delta t = 0.1$. The times shown are $t=300, 500, 700, 900, 1100, 1300$ (from left to right, top to bottom). The colors are scaled from $\sphi = -0.2$ to $\sphi = 0.23$. See \cref{ss:phase_transformation} for more details.}
	\label{fig:phase_transition}
\end{figure}

\subsection{Comparison of degrees of freedom}
In this section, we compare the degrees of freedom from this work against the $C^0$ IPDG discretization from \cite{DS:23:C0PFC}, and the LDG discretization from \cite{guo2016ldg}, using the different meshes in \cref{ss:accuracy,ss:crystal_formation,ss:benchmark,ss:grain_grows}. The $C^0$ IPDG discretization from \cite{DS:23:C0PFC} uses the following discrete spaces
\begin{align*}
    V_h&:=\{v \in C(\bar{\Omega}): \, v|_K\in P_1(K), \forall K \in \mathcal{T}_h\},\\
    Z_h&:=\{v \in C(\bar{\Omega}): \, v|_K\in P_2(K), \forall K \in \mathcal{T}_h\},
\end{align*}
meaning that $V_h$ and $Z_h$ contain continuous first and second-degree polynomials, respectively.
On the other hand, the LDG discretization from \cite{guo2016ldg} uses the following discrete spaces
\begin{align*}
    V_h&:=\{\varphi: \, \varphi|_K\in P^k(K), \forall K \in \mathcal{T}_h\},\\
    \Sigma^d_h&:=\{\Phi = (\phi_1,\ldots, \phi_d)^T: \, \phi_l|_K\in P^k(K),\ l=1,\dots, d, \forall K \in \mathcal{T}_h\},
\end{align*}
meaning that scalar and vector-valued discontinuous polynomials of degree $k$ are used.
While both \cite{DS:23:C0PFC} and \cite{guo2016ldg} use triangular meshes, and the current work uses quadrilateral meshes, we compare the degrees of freedom for the case when the meshes from the current work are replaced by triangular ones, using the same spacing in both the $x$ and the $y$ direction. The results are shown in \cref{tab:dof_compare}, where we observe significantly lower (approximately 12\% of the LDG case) number of globally coupled degrees of freedom for the EDG spaces in \cref{eq:M_space}. 
\begin{table}[!h]
    \centering
    \begin{tabular}{l|c|c|c|c}
        Mesh & $C^0$ IPDG & LDG $k=1$ & HDG & EDG\\
   Sec. 5.1 1st level  & 11520 & 55296 & 41472 & 6912 \\
   Sec. 5.1 2nd level  & 46080 & 221184 & 165888 & 27648 \\
   Sec. 5.1 3rd level  & 184320 & 884736 & 663552 & 110592 \\
   Sec. 5.1 4th level  & 737280 & 3538944 & 2654208 & 442368 \\
   Sec. 5.1 5th level  & 2949120 & 14155776 & 10616832 & 1769472 \\
   Sec. 5.2  & 1223600 & 5873280 & 4404960 & 734160 \\
   Sec. 5.3  & 329218 & 1572864 & 1182720 & 198147   \\
   Sec. 5.4  & 808020 & 3878496 & 2908872 & 484812 \\
    \end{tabular}
    \caption{Globally coupled degrees of freedom comparison between \cite{DS:23:C0PFC}, \cite{guo2016ldg} and the current manuscript for the lowest polynomial order.}
    \label{tab:dof_compare}
\end{table}

\subsection{Energy stability}\label{ss:energy_stability}

In the final numerical experiment, we show the effect the stabilization parameters have on energy stability. In \Cref{lem:full_energy_stability} it was shown that $\tau_1,\tau_3 > 0$, $\tau_4 = 2\tau_1$, and $\tau_2 = \tau_1$ are necessary to obtain energy stability. In this section, we will show that the energy might increase if these relations are violated. 

We consider a slightly modified version of the example from \cref{ss:phase_transformation}. The computational domain is $(0. 2\pi)^2$ with periodic boundary conditions. The mesh consists of $24 \times 24$ quadrilaterals, and the model parameter is $\varepsilon = 0.5,$. The time stepping is $\Delta t = 1.0$, and we used bilinear basis functions. First we performed the calculations with the choice $\tau_1 = \tau_2 = \tau_3 = \frac12\tau_4 = 10$, which satisfies the conditions of \Cref{lem:full_energy_stability}. The resulting energy curve can be seen in the top left of \cref{fig:energy_stability}. The other three simulations used stability parameters violating the conditions of \cref{lem:full_energy_stability}:
\begin{enumerate}
	\item $\tau_1 = \tau_2 = 1$, $\tau_3 = -10$, $\tau_4 = -30$, the resulting energy curve is shown in the top right of \cref{fig:energy_stability},
	\item $\tau_1 = \tau_2 = \tau_3 = 1$, $\tau_4 = 3000$, the resulting energy curve is shown in the bottom left of \cref{fig:energy_stability},
	\item $\tau_1 = \tau_2 = \tau_3 = 10$, $\tau_4 = 3000$, the resulting energy curve is shown in the bottom right of \cref{fig:energy_stability}.
\end{enumerate}

We can see that these incorrect choices all result in an increase of the energy.

\begin{figure}
	\centering
	\includegraphics[width=0.48\textwidth]{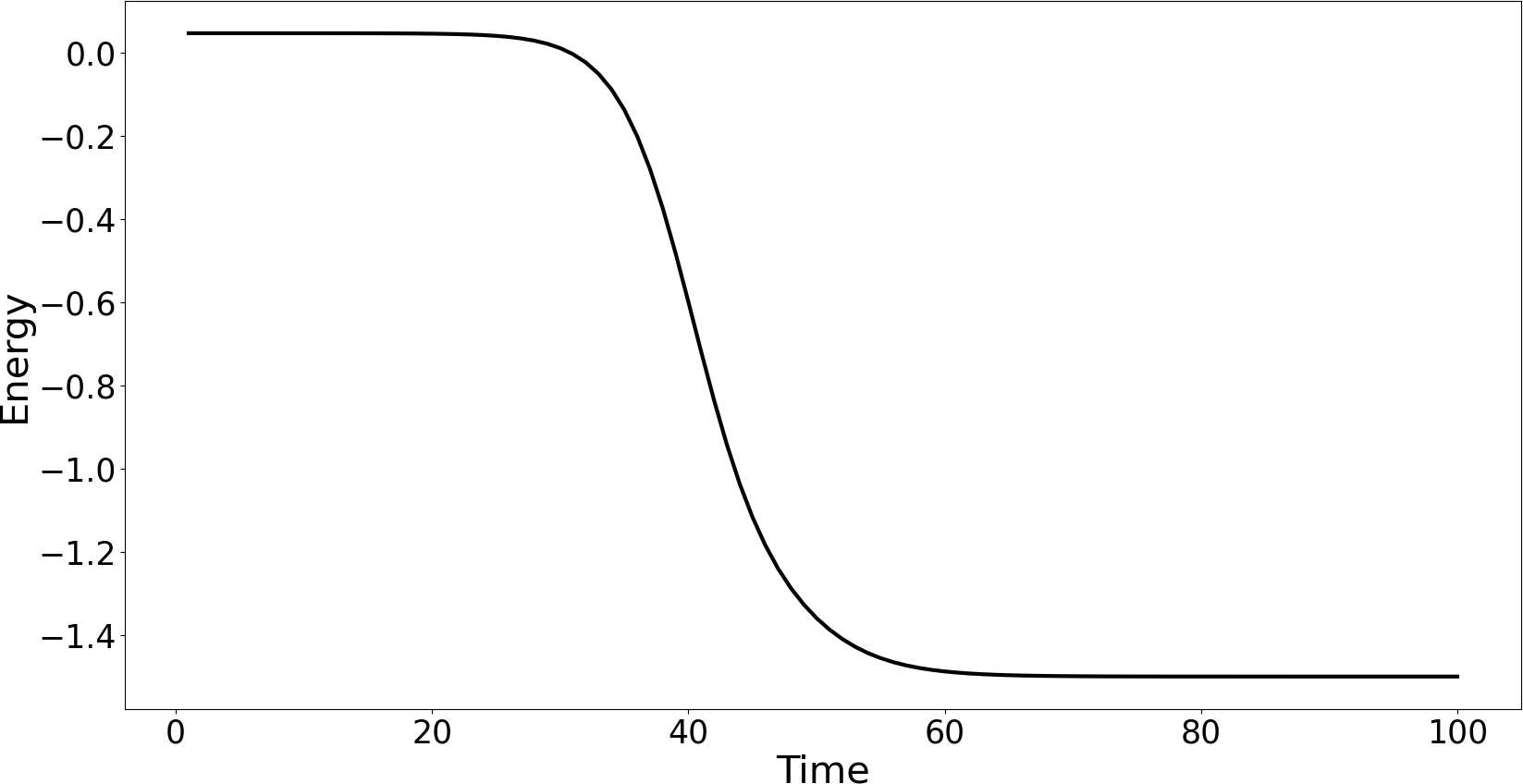}
	\includegraphics[width=0.48\textwidth]{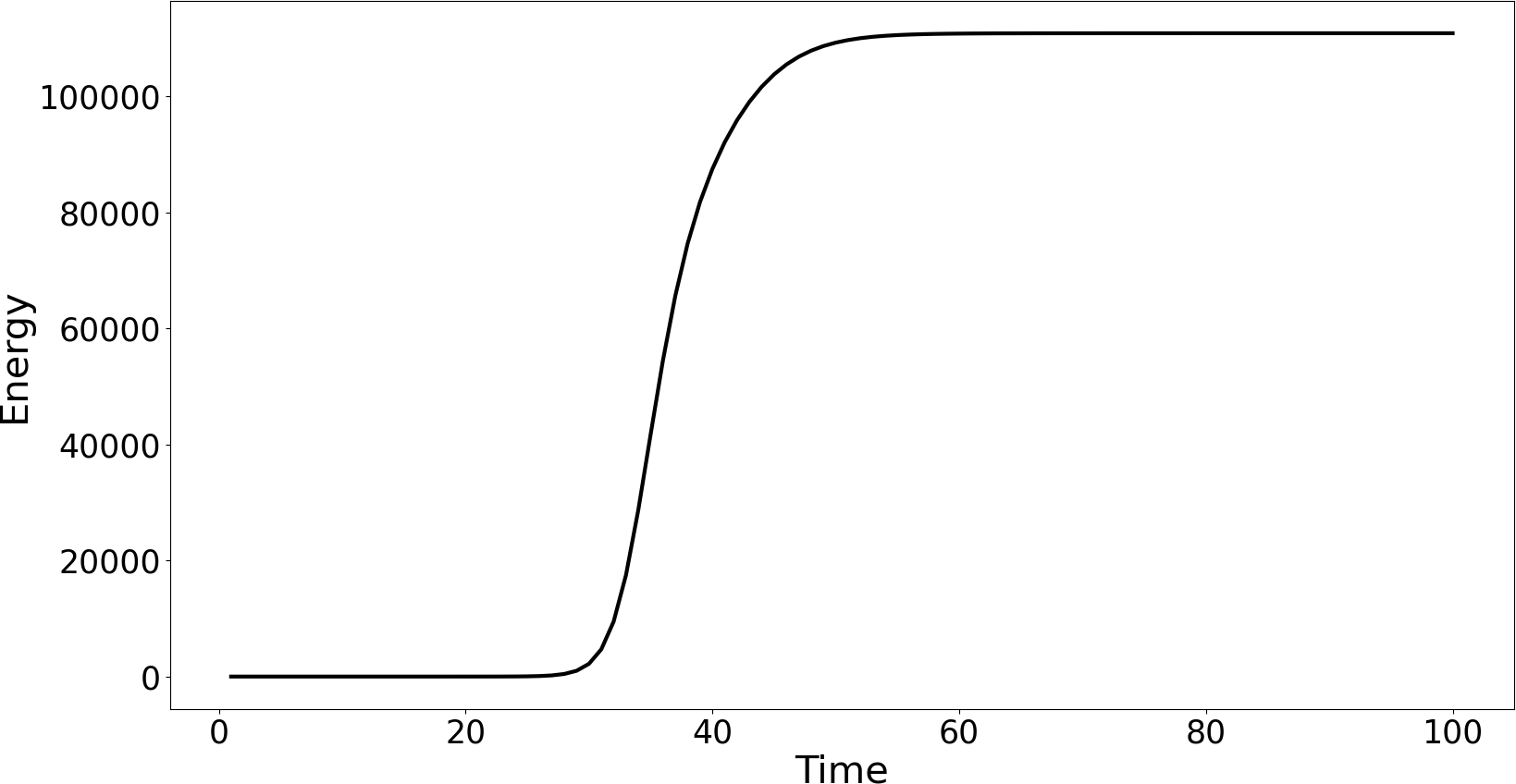}\\
	\includegraphics[width=0.48\textwidth]{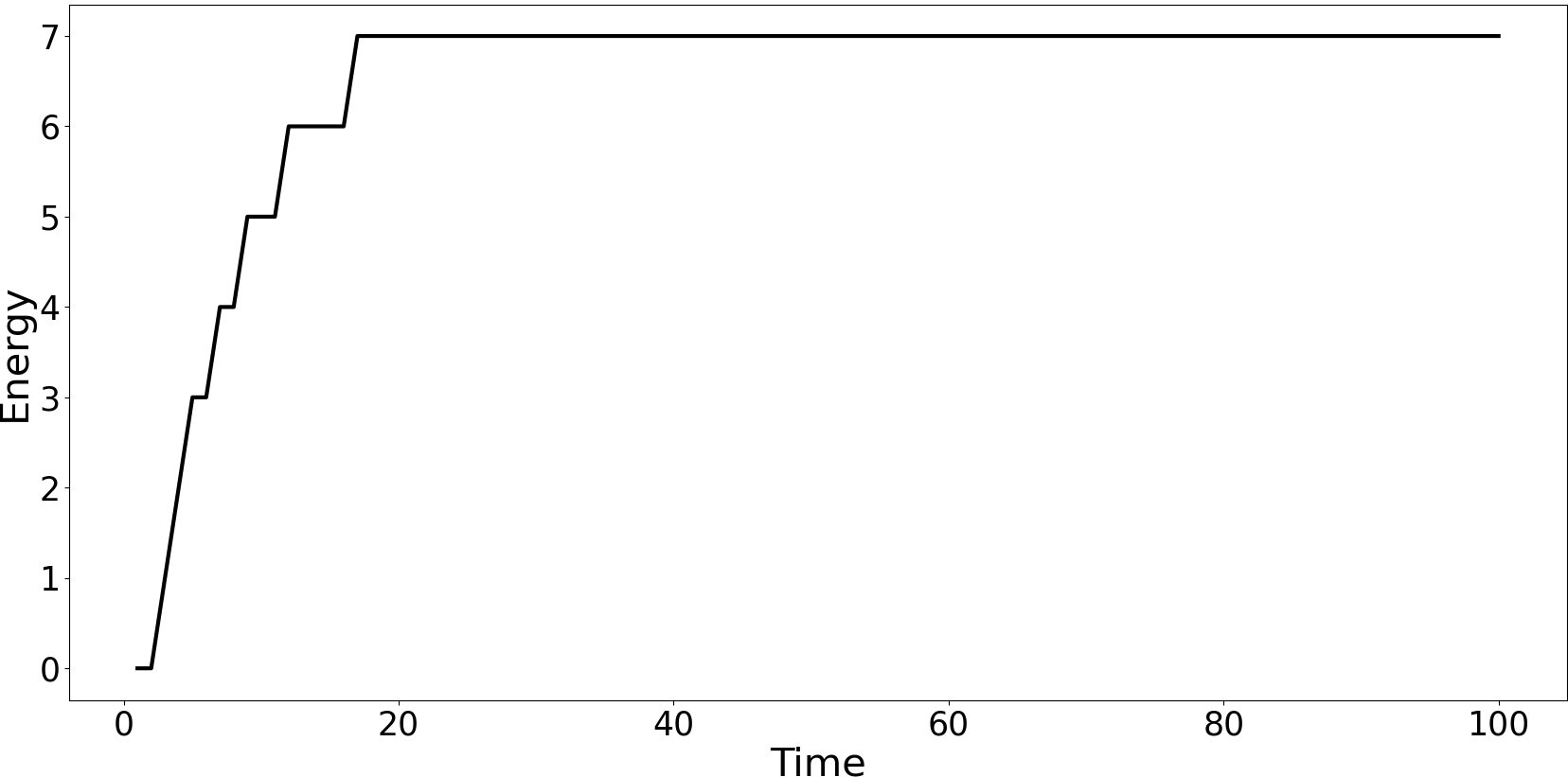}
	\includegraphics[width=0.48\textwidth]{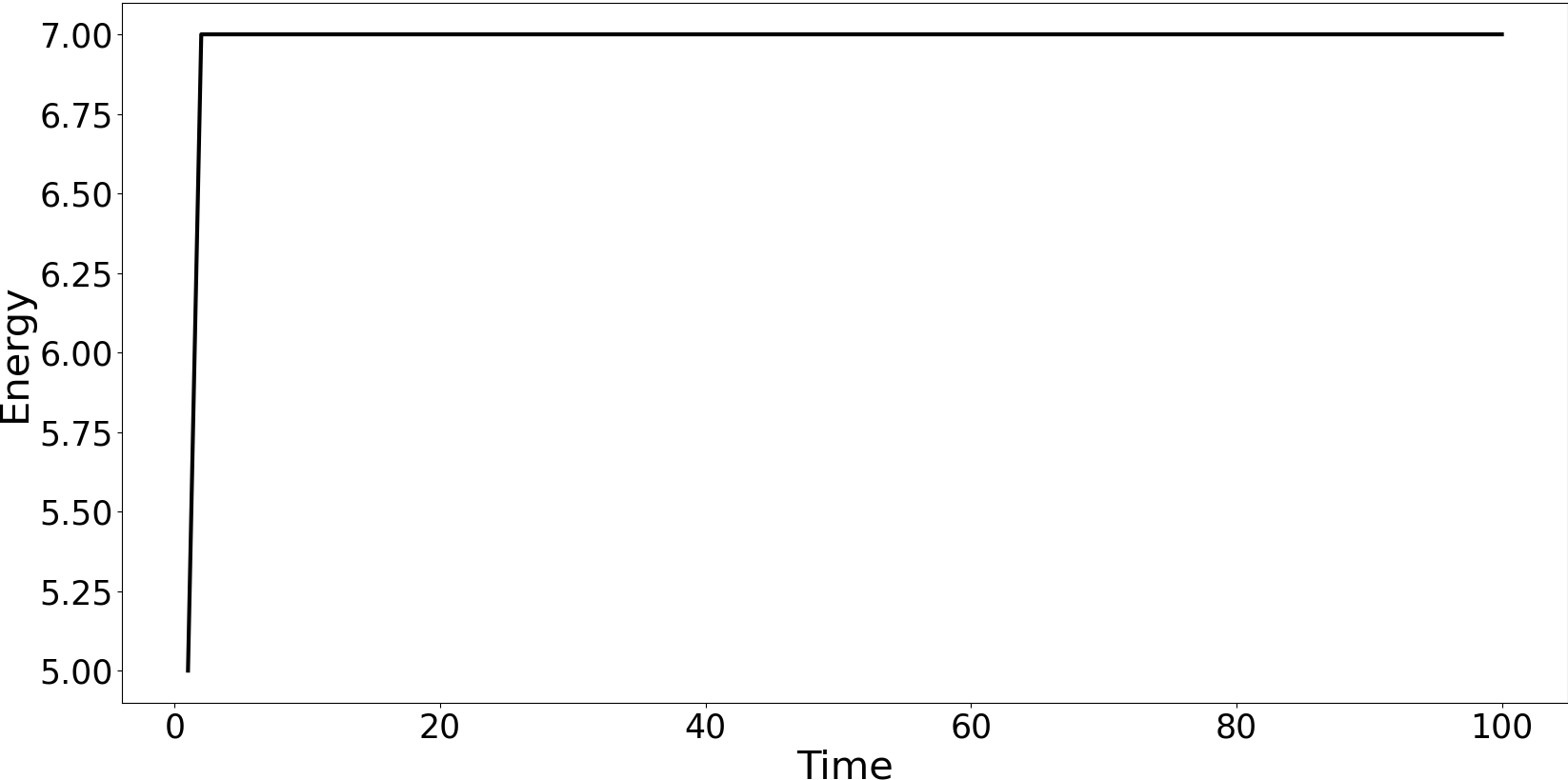}
	\caption{The effect of the stabilization terms on the energy stability of the proposed method. See \cref{ss:energy_stability} for more details.}
	\label{fig:energy_stability}
\end{figure}

\section{Conclusions and future work}
\label{sec:conclusions}
This paper presents a first-order convex splitting hybridizable/embedded discontinuous Galerkin method for the phase field crystal equation. Under a certain relation between the stabilization parameters, we have established the energy stability of the scheme, and prove the existence and uniqueness of a solution for the proposed fully discrete scheme, with variable mobility. The scheme's performance is presented through the results of some numerical examples. Future work includes convergence analysis and error estimates, as well as extending the method to a higher-order scheme in time that is unconditionally energy stable as well. Moreover, we would like to extend this approach to models that involve higher-order derivatives, such as the eighth-order phase field crystal model~\cite{jaatinen2009thermodynamics,jaatinen2010eighth}.

\bibliographystyle{elsarticle-num-names}
\bibliography{hdg_for_pfc}
\end{document}